\documentclass[a4paper, 12pt]{amsart}

\usepackage{amssymb,latexsym}

\usepackage{xcolor}

\usepackage{tikz}
\usetikzlibrary{calc}

\usepackage{graphicx}

\usepackage{enumerate}

\usepackage[T1]{fontenc}
\usepackage{hyperref}
\hypersetup{
    colorlinks=true,
    linkcolor=blue,
    filecolor=blue,      
    urlcolor=blue,
    citecolor=blue,
    pdftitle=,
    pdfpagemode=FullScreen,
    }

 \usepackage[french,english]{babel}

\makeatletter

\@namedef{subjclassname@2010}{

  \textup{2010} Mathematics Subject Classification}

\makeatother

\newtheorem{thm}{Theorem}[section]
\newtheorem{Con}[thm]{Conjecture}

\newtheorem{lem}[thm]{Lemma}
\newtheorem{pro}[thm]{Proposition}
\theoremstyle{definition}

\newtheorem{rem}{Remark}

\numberwithin{equation}{section}

\renewcommand{\le}{\leqslant}
\renewcommand{\leq}{\leqslant}

\renewcommand{\geq}{\geqslant}
\newcommand{\A}{\mathcal{A}}
\newcommand{\Ld}{\mathcal{L}_d}
\newcommand{\Lrand}{\mathcal{L}_{\textup{rand}}}
\newcommand{\X}{\mathbb{X}}

\newcommand{\ex}{\mathbb{E}}
\newcommand{\re}{\textup{Re}}
\newcommand{\im}{\textup{Im}}
\newcommand{\pr}{\mathbb{P}}
\newcommand{\ep}{\varepsilon}
\newcommand{\mc}{\mathcal}

\newcommand{\Df}{\mathcal{D}(x)}
\newcommand{\D}{\mathcal{D}}
\newcommand{\Ds}
{\widetilde{\mathcal{D}}(x)}

\newcommand{\mb}{\mathbb}

\newcommand{\newabstract}[1]{%
  \par\bigskip
  \csname otherlanguage*\endcsname{#1}%
  \csname captions#1\endcsname
  \item[\hskip\labelsep\scshape\abstractname.]
}

\begin{document}

\baselineskip=17pt

\title[]{Real zeros of $L'(s, \chi_d)$}

\author{Youness Lamzouri}
\author{Kunjakanan Nath}

\address{Universit\'e de Lorraine, CNRS, IECL,  and  Institut Universitaire de France,
F-54000 Nancy, France}

\email{youness.lamzouri@univ-lorraine.fr}

\address{Institut \'Elie Cartan de Lorraine, Universit\'e de Lorraine, CNRS, F-54000 Nancy, France}
\email{kunjakanan@gmail.com}

\begin{abstract} 
In 1990, Baker and Montgomery conjectured that $L'(s,\chi_d)$ has $\asymp \log\log |d|$ real zeros in the interval $[1/2,1]$ for almost all fundamental discriminants $d$. The study of these zeros was motivated by their connection to real zeros of Fekete polynomials and to sign changes of the character sums $\sum_{n\leq x}\chi_d(n)$. Recent work of Klurman, Lamzouri, and Munsch shows that the number of such zeros is $\gg (\log\log |d|)/(\log\log\log\log |d|)$ for almost all $d$, thereby establishing the conjectured lower bound up to the factor $\log\log\log\log |d|$. In this paper, we prove that for almost all fundamental discriminants $d$, $L'(s,\chi_d)$ has at most $(\log\log |d|)(\log\log\log |d|)$ real zeros in $[1/2,1]$, thus resolving the Baker-Montgomery conjecture up to a factor of $\log\log\log |d|$. We also give a quantitative upper bound on the exceptional set of discriminants. Furthermore, we show, conditionally on certain natural assumptions, that $100\%$ of these zeros lie away from $1/2$.

\end{abstract}



\maketitle


\section{Introduction}

Understanding the location and distribution of zeros of derivatives of $L$-functions has important and deep applications to the \emph{horizontal}  and \emph{vertical} distributions of zeros of $L$-functions. One of the earliest and most striking links between the zeros of $\zeta'(s)$ (where $\zeta(s)$ is the Riemann zeta function) and the Riemann Hypothesis (RH) is Speiser's Theorem \cite{Sp35}, which states that RH is equivalent to the assertion that $\zeta'(s)$ has no zeros to the left of the critical line. This was quantified by Levinson and Montgomery \cite{LeMo74}, and is the basis of Levinson's method which produces  one third of the zeros of $\zeta(s)$  on the critical line. Furthermore, the works of Soundararajan \cite{So98}, and Radziwi\l\l \ \cite{Rad14} show that the horizontal distribution of the zeros of $\zeta'(s)$ is also related to the vertical distribution of the zeros of $\zeta(s).$

\subsection{The Baker-Montgomery conjecture} In \cite{BaMo89}, Baker and Montgomery studied the real zeros of $L'(s, \chi_d)$ on $[1/2, 1]$, where $\chi_d$ is the primitive quadratic character attached to the fundamental discriminant $d$, and $L(s, \chi_d)$ is the associated Dirichlet $L$-function. Baker and Montgomery's motivation was to study real zeros of Fekete polynomials, and sign changes of quadratic character sums. Let $F_d(z):=\sum_{n=1}^{|d|-1}\chi_d(n) z^n$ be the Fekete polynomial associated to $d$. Fekete observed that if $F_d$ does not vanish on $(0, 1)$ then $L(s, \chi_d)>0$ for all $s\in (0, 1)$, which in particular implies Chowla's conjecture that $L(1/2, \chi_d)\neq 0$, and refutes the existence of a possible Siegel zero. This  
follows from the following identity, obtained by a familiar inverse Mellin transform
\begin{equation}\label{Eq.MellinIdentity}L(s,\chi_d)\Gamma(s)=\int_0^1 \frac{(-\log u)^{s-1}}{u}\frac{F_d(u)}{1-u^d}du, \text{ for } \re(s)>0.\end{equation}
Fekete conjectured that $F_d$ does not vanish on $(0, 1)$ if $|d|$ is large enough, but this was disproved shortly afterwards by P\'olya \cite{Pol19}, for a positive proportion of fundamental discriminants $d$. In \cite{BaMo89}, Baker and Montgomery proved that Fekete's hypothesis is false for $100\%$ of fundamental discriminants. In fact, they proved the stronger result that for any fixed positive integer $K$, $F_d$ has at least $K$ zeros in $(0, 1)$ for almost all fundamental discriminants $d$. Baker and Montgomery's approach consists in relating zeros of $F_d$ on $(0,1)$ to sign changes of $\frac{L'}{L}(s, \chi_d)$ on $(1/2, 1)$ via the following identity which is obtained from  \eqref{Eq.MellinIdentity} by differentiating  with respect to $s$:   
\begin{equation}\label{Eq.MellinIdendity2}
L(s,\chi_d)\Gamma(s)\left(\frac{L'(s,\chi_d)}{L(s,\chi_d)}+\frac{\Gamma'(s)}{\Gamma(s)}\right)=
\int_{0}^{\infty} F_d(e^{-t})(1-e^{-|d|t})^{-1}t^{s-1}(\log t)dt.    
\end{equation}
 Indeed, if the left-hand side of \eqref{Eq.MellinIdendity2} has $K$ sign changes in  $(1/2, 1)$ (which implies in particular that $L'(s, \chi_d)$ has $K$ zeros in this interval) then $F_d$ has at least $K$ zeros on $(0,1)$ by a lemma of a real analysis (see Lemma 4 of \cite{BaMo89}).

Let $R_d(\sigma_1,\sigma_2)$ be the number of real zeros of $L'(s, \chi_d)$ on the interval $[\sigma_1, \sigma_2]$.  Based on a heuristic argument inspired by their construction, Baker and Montgomery made the following conjecture.
\begin{Con}[\cite{BaMo89}, Baker-Montgomery]\label{Conj.BaMo}
For almost all fundamental discriminants $d$, we have 
$$ R_d\left(\frac12, 1\right) \asymp \log\log |d|.$$
\end{Con}
In \cite{KLM24}, Klurman, Lamzouri, and Munsch proved that for almost all fundamental discriminants $d$ we have 
\begin{equation}\label{Eq.LowerBoundKLM}
R_d\left(\frac{1}{2}, 1\right) \gg \frac{\log\log |d|}{\log_4|d|},
\end{equation}
where here and throughout $\log_k$ denotes the $k$-th iterate of the natural logarithm function. This comes close of establishing the lower bound in Conjecture \ref{Conj.BaMo}.

Baker and Montgomery \cite{BaMo89} (and later Conrey,  Granville, Poonen, and Soundararajan \cite{CGPS00}) made a similar conjecture about the number of real zeros of $F_d$ on $(0, 1)$, predicting that it should be $\asymp \log\log|d|$ for almost all $d$. Klurman, Lamzouri, and Munsch \cite{KLM24} established an analogous ``localized'' version of the lower bound \eqref{Eq.LowerBoundKLM} in this case, using appropriate variants of  \eqref{Eq.LowerBoundKLM} concerning oscillations of $L'(s, \chi_d)$, coupled with a concentration result for the distribution of $L(s, \chi_d)$ in the vicinity of $1/2$.   However, the only partial result towards the conjectured upper bound for the number of real zeros of $F_d$ was established in  \cite{KLM24} and states that for at least $x^{1-\ep}$ fundamental discriminants $|d|\leq x$, $F_d$ has at most $O(x^{1/4+\ep})$ zeros in $(0,1)$. This breaks the $O(\sqrt{x})$ bound which holds for all Littlewood polynomials by a result of  Borwein, Erd\'elyi, and K\'os \cite{BEK99}, but is very far from the conjectured  $\log\log x$ bound. 

In this paper, we focus on the upper bound in Conjecture \ref{Conj.BaMo}. For convenience, as in previous works on the moments and non-vanishing of $L(1/2, \chi_d)$, we restrict the modulus $d$ to be of the form $8m$ where $m$ is squarefree and odd.  However, our methods would apply to fundamental discriminants in any fixed arithmetic progression. Here and throughout, we define
$$\Df:=\{d=8m :  m \textup{ is squarefree and odd, and } x/2\leq m\leq x\}.$$ 
Note that $|\Df|\asymp x$.

Our main result shows that $R_d(1/2, 1)\ll (\log\log x) (\log\log \log x)$ for 100\% of fundamental discriminants $d\in \D(x)$, thus  resolving the Baker-Montgomery conjecture, up to a factor of $\log\log\log x$. 

\begin{thm}\label{Thm:Main}
For all discriminants $d\in \D(x)$, with the exception of a set of cardinality $\ll x\log_3 x/\sqrt{\log\log x},$ we have 
$$ 
R_d\left(\frac12, 1\right)\ll (\log\log x)(\log_3 x).
$$
\end{thm}

Our proof begins by splitting the interval $[1/2,1]$ into two subintervals $I_1=[1/2,\,1/2+1/H(x)]$ and $I_2=[1/2+1/H(x),\,1]$, where $H(x)=(\log x\,\log_3 x)/\log_2 x$. The interval $I_1$ corresponds to the region very close to the central point, while $I_2$ lies away from $1/2$. We first describe our strategy for bounding the number of zeros of $L'(s,\chi_d)$ on $I_2$. Using zero-density estimates, we show that $\frac{L'}{L}(s,\chi_d)$ is analytic in an open disc containing $I_2$, for almost all\footnote{Here and throughout, we say that almost all $d\in \mathcal{D}(x)$ satisfy the property $P$ if $|\{d\in \mathcal{D}(x)\colon d\: \text{has property $P$}\}|\sim |\Df|$ as $x\to \infty$. } $d\in \D(x)$. We then cover $I_2$ by a union of $J\asymp\log\log x$ smaller discs $\{D_j\}_{j\le J}$, and apply Jensen’s formula to bound the number of zeros inside each disc. The main advantage of working with $\frac{L'}{L}(s,\chi_d)$, rather than directly with $L'(s,\chi_d)$, is the crucial fact that, after suitable normalization, $\frac{L'}{L}(s,\chi_d)$ admits a limiting distribution that becomes Gaussian as $s\to1/2$. This allows us to exploit information on both the large and small values of $\frac{L'}{L}(s,\chi_d)$ in a slightly larger disc containing $D_j$, which in turn yields bounds for the number of zeros in $D_j$ via Jensen’s formula. This approach, however, breaks down on $I_1$, since it is not known unconditionally that for almost all fundamental discriminants $d$,  $L(s,\chi_d)\neq 0$ in a small disc containing this interval\footnote{This is why assumptions on low-lying zeros of $L(s,\chi_d)$ are required to prove the conditional Theorem~\ref{Thm:ConditionalZerosNear1/2}.}. Consequently, we instead study $L'(s,\chi_d)$ itself on a small disc $\widetilde D_0$ centered at $s_0=1/2+1/H(x)$ and containing $I_1$. To apply Jensen’s formula and bound the number of zeros of $L'(s,\chi_d)$ in $\widetilde D_0$, we must control the large values of $|L'(s,\chi_d)|$ on the boundary of a slightly larger disc $\widetilde D_1$, as well as its small values at the center $s_0$. We achieve the first goal by bounding the second moment of $\max_{z\in\partial\widetilde D_1}|L'(z,\chi_d)|$. For the second, we use the identity $|L'(s_0,\chi_d)|=|\frac{L'}{L}(s_0,\chi_d)|\exp(\log|L(s_0,\chi_d)|)$ and exploit information on the joint distribution of $\frac{L'}{L}(s_0,\chi_d)$ and $\log L(s_0,\chi_d)$.

\subsection{The location of real zeros of $L'(s, \chi_d)$.}
The real zeros of $L'(s,\chi_d)$ constructed by the authors of \cite{KLM24} all lie in the interval $[1/2+1/(\log x)^{1/5},\,1]$. The exponent of $\log x$ was not optimized in \cite{KLM24}, since this was not required to establish \eqref{Eq.LowerBoundKLM}. Nevertheless, their method should yield the same lower bound $\log\log x/\log_4 x$ for the number of zeros of $L'(s,\chi_d)$ in the interval $[1/2+1/(\log x)^{1/2},\,1/2+1/(\log x)^{\alpha}]$, for almost all $d\in\D(x)$, where $0\leq \alpha<1/2$ is fixed. Using our approach, one can go further and show that, for any fixed $0\leq \alpha<1$ and for almost all $d\in\mathcal{D}(x)$, the number of zeros of $L'(s,\chi_d)$ in the interval $[1/2,\,1/2+1/(\log x)^{\alpha}]$ equals $(\log\log x)(\log\log\log x)^{\theta}$ for some $|\theta|\leq 1$.

Assuming the Riemann Hypothesis (RH), Soundararajan \cite{So98} proved that a  positive proportion of the zeros of $\zeta'(s)$ up to height $T$ are in the strip $1/2\leq \re(s)\leq 1/2+ 3/\log T$. One can ask a similar question in our context: for a ``generic'' fundamental discriminant $d$, does a positive proportion of the real zeros of $L'(s, \chi_d)$ in $[1/2, 1] $ lie within distance $c/\log x$ (or even a bit further) of the central point? We show that this is not the case, conditionally on the following natural assumptions on the low lying zeros of $L(s, \chi_d)$:
\begin{itemize}
   \item \textbf{Assumption 1 (Weak GRH for almost all $d$)} 
   For almost all $d\in \Df$, the zeros of $L(s, \chi_d)$ in the rectangle $1/2-1/\log x \leq \re(s)\leq 1$ and $|\im(s)|\leq \sqrt{\log\log x}/\log x$ all lie on the critical line. 
   \item \textbf{Assumption 2 (Low Lying Zeros hypothesis)}
Let $0<\delta<1/2$. For a fundamental discriminant $d$ we let $\gamma_{\text{min}}(d)=\min\{|\gamma| : L(\beta+i\gamma, \chi_d)=0, \text{ and } 0<\beta<1\}.$ Then we have 
$$ \lim_{x\to \infty}\frac{1}{|\Df|}\#\left\{ d\in \D(x) : \gamma_{\text{min}}(d) \leq \frac{1}{(\log\log x)^{\delta}\log x} \right\}=0.$$

\end{itemize}

\begin{thm}\label{Thm:ConditionalZerosNear1/2} Suppose that Assumption 1 holds, and that Assumption 2 holds with constant $0<\delta<1/2$. Let $\nu(x)=(\log\log x)^{1/2-\delta}/\log_3 x$. For almost all $d\in \D(x)$ we have  
\begin{equation}\label{Eq.MostZerosRight}
R_d\left(\frac12,\frac12+ \frac{\nu(x)}{\log x}\right)= o\left(R_d\left(\frac12, 1\right)\right) \quad \text{as $x\to \infty$}.
 \end{equation}
\end{thm} 
\begin{rem}
Since the conductor of our family is $\asymp x$, the average spacing of the zeros of $L(s, \chi_d)$ is $\asymp 1/\log x$, and hence we expect that Assumption 2 holds with any $\delta>0$.  In fact, this assumption follows from GRH and the one level density conjecture of Katz
and Sarnak \cite{KaSa99}, which predicts that 
\begin{equation}\label{Eq.ZeroDensityConjecture} \lim_{x\to \infty} \frac{1}{|\Df|}\sum_{d\in \Df} \sum_{\substack{\rho=1/2+i\gamma\\ L(\rho, \chi_d)=0}} \phi\left(\frac{\gamma\log x}{2\pi}\right)= \int_{-\infty}^{\infty} \phi(u) \left(1-\frac{\sin(2\pi u)}{2\pi u}\right) du, 
\end{equation}
for any real even Schwartz class test function, whose Fourier transform has compact support. A stronger form of Assumption 2, where $(\log\log x)^{\delta}$ is replaced by any positive function $\nu(x)$ such that $\nu(x)\to \infty$ as $x\to \infty$, was used by  Hough \cite{Ho14} to prove a conjecture of Keating and Snaith \cite{KeSn00}, which is the  analogue of Selberg's central limit theorem for the distribution of $\log L(1/2, \chi_d)$ as $d$ varies in $\Df$. %
\end{rem}

Finally, we remark that using the methods of this paper, all our results can be extended to the orthogonal family of $L$-functions attached to Hecke cusp forms of weight $k$ for the full modular group, as $k\to\infty$.

\subsection*{Notation} We will use standard notation in this paper. However, for the convenience of readers, we would like to highlight a few of them. Expressions of the form ${f}(x)=O({g}(x))$, ${f}(x) \ll {g}(x)$, and ${g}(x) \gg {f}(x)$ signify that $|{f}(x)| \leq C|{g}(x)|$ for all sufficiently large $x$, where $C>0$ is an absolute constant. A subscript of the form $\ll_A$ means the implied constant may depend on the parameter $A$. The notation ${f}(x) \asymp {g}(x)$ indicates that ${f}(x) \ll {g}(x) \ll {f}(x)$. Next, we write $f(x)=o(g(x))$ if $\lim_{x\to \infty}f(x)/g(x)=0$.

\subsection*{Organization of the paper} The paper is organized as follows. In Section \ref{Section:MeanValues} we gather together several mean value estimates involving quadratic characters. In Section \ref{Section:ApproxLogDerL} we use ideas of Selberg and zero density estimates to approximate $-\frac{L'}{L}(s, \chi_d)$ by short Dirichlet polynomials, for almost all $d\in \Df$, once $\re(s)\geq 1/2+\nu(x)/\log x$, where $\nu$ is any positive function such that $\nu(x)\to \infty$ as $x\to \infty$. In Section \ref{Section:Discrepancy} we establish a bound for the discrepancy between the distribution of $-\frac{L'}{L}(s, \chi_d)$ (normalized by $1/(s-1/2)$) and that of a corresponding random model, uniformly in the range $1/2+\nu(x)/\log x\leq s\leq 1$. In Section \ref{Section. Thm main away from the central point}, we establish Theorem \ref{Thm:RealZerosAway1/2}, which counts the number of real zeros of $L^\prime(s, \chi_d)$ away from the central point. Next, in Section \ref{Sec. Thm main near the central point}, we prove Theorem \ref{Thm:ZerosNearCentral}, which bounds the number of zeros of $L^\prime(s, \chi_d)$ near $1/2$. Theorem \ref{Thm:Main} then follows from combining Theorems \ref{Thm:RealZerosAway1/2} and \ref{Thm:ZerosNearCentral}. Finally, in Section \ref{Section:MomentsNearHalf}, we establish our conditional result Theorem \ref{Thm:ConditionalZerosNear1/2}.


\section{Mean values of Dirichlet polynomials with quadratic characters}\label{Section:MeanValues}

In this section we gather together several  mean value estimates with quadratic characters. The first is an “orthogonality relation” for the family $\Df$.
\begin{lem}\label{Lem.Orthogonality}
For all $n\leq x$ we have 
\begin{equation}\label{Eq.Ortho1}
 \frac{1}{|\Df|}\sum_{d\in \Df} \chi_d(n)= \begin{cases} \displaystyle{\prod_{\substack{p\mid n \\ p>2}} \left(\frac{p}{p+1}\right)}+ O(x^{-1/5}) & \text{ if } n \text{ is a square},\\
O(x^{-1/5}) & \text{ otherwise}. \\
\end{cases}
\end{equation}
\end{lem}
\begin{proof}
    This is a special case of Lemma 2.3 of \cite{Ho14}, upon taking $\delta= 1$ and choosing $\gamma(\delta)=1/5$ therein, which is admissible. 
\end{proof}

Next, we state the following large-sieve type result from \cite{KLM24}, which is a consequence of the above lemma.
\begin{lem}[Lemma 3.2 of \cite{KLM24}]\label{Lem: LargeSieveKLM}
Let $\{a(p)\}_{p}$ be a sequence of real numbers indexed by the primes. Let $x$ be large and $2\leq y\leq z$ be real numbers. Then for all positive integers $k$ such that $1\leq k\leq \log x/(5\log z)$ we have
\begin{equation}\label{LargeSieve1}
\sum_{d\in \mc{D}(x)}\Big|\sum_{y\leq p\leq z}a(p)\chi_d(p)\Big|^{2k}\ll x\left(k\sum_{y\leq p\leq z}a(p)^2\right)^k+x^{5/8}\left(\sum_{y\leq p\leq z}|a(p)|\right)^{2k}.
\end{equation}
\end{lem}
We will need the following result on the second moment of real character sums, which was established by Armon \cite{Ar99}. 

\begin{lem}[Theorem 2 of \cite{Ar99}]\label{Lem:Armon} For all real numbers $x\geq 2$ and $y\geq 1$ we have 
$$ \sum_{d\in \D(x)}\left|\sum_{n\leq y}\chi_d(n)\right|^2\ll x y \log x.$$
\end{lem}

We now introduce the probabilistic random model corresponding to the family $\{\chi_d\}_{d\in \D(x)}$. Let $\{\mb{X}(p)\}_{p\: \text{prime}}$ be a sequence of independent random variables defined as: $\mb{X}(2)=0$; and for $p>2$,  $\X(p)$ takes the values $\{-1, 0, 1\}$  with probabilities
\begin{align*}
   \mb{P}\big(\mb{X}(p)=1\big)=\mb{P}\big(\X(p)=-1)=\dfrac{p}{2(p+1)}, \quad \text{and} \quad \mb{P}\big(\X(p)=0\big)=\dfrac{1}{p+1}.
\end{align*}
We extend the $\X(p)$ multiplicatively by setting $\X(n)=\X(p_1)^{a_1} \cdots \X(p_k)^{a_k}$ if $n$ has the prime factorization $n=p_1^{a_k}\cdots p_k^{a_k}$. Then one can write \eqref{Eq.Ortho1} as 
\begin{equation}\label{Eq.Orthogonality}
\frac{1}{|\Df|}\sum_{d\in \Df} \chi_d(n)= \ex(\X(n))+ O(x^{-1/5}),
\end{equation} 
for all $n\leq x$.
As a consequence, we establish the following lemma.
\begin{lem}\label{lem.AsympMoments}
Let $C>0$ be a fixed constant. Let $b(n)$ be real numbers such that $|b(n)|\leq C $ for all $ n \geq 1$. 
Then uniformly for $x\geq Y\geq  2$ and all positive integers $k\leq \log x/\log Y$ we have 
$$
 \frac{1}{|\Df|}\sum_{d\in \Df}\left(\sum_{n \leq Y} b (n) \chi_d(n)\right)^{k}  = \ex\bigg[\bigg(\sum_{n \leq Y} b (n)\X(n)\bigg)^{k}\bigg]  + O \big(x^{-1/5}(CY)^{k} \big),$$
where the implicit constant in the error term is absolute.
\end{lem}
\begin{proof}
    We have
\begin{align*}
   \frac{1}{|\Df|} \sum_{d\in \Df}\left(\sum_{n \leq Y} b (n) \chi_d(n)\right)^{k}   
& = \frac{1}{|\Df|}\sum_{d\in \Df} \bigg( \sum_{n_1, n_2, \dots, n_k \leq Y}   \prod_{i=1}^{k} b (n_i)\chi_d(n_i)\bigg) \\
& =\sum_{n_1, n_2, \dots, n_k \leq Y}    \prod_{i=1}^{k} b (n_i)  \frac{1}{|\Df|}\sum_{d\in \Df} \chi_d\bigg(\prod_{i=1}^{k} n_{i}\bigg).
\end{align*}
By \eqref{Eq.Orthogonality} and the fact that $|b(n)|\leq C$ for all $n\geq 1$, this sum equals
\begin{align*}
& \sum_{n_1, \dots, n_k\leq Y}   \prod_{i=1}^k b (n_i)  \ex\Big(\prod_{i=1}^k\X(n_i)\Big) + O\left(x^{-1/5}(CY)^{k}\right)\\
 & = \ex \bigg[\bigg(\sum_{n \leq Y} b (n)\X(n)\bigg)^{k}   \bigg] +O\left( x^{-1/5} (CY)^{k}\right),
\end{align*}
as desired.
\end{proof}
We end this section by proving upper bounds for the moments of certain quadratic character sums supported on prime powers.

\begin{lem}\label{Lem.LargeSieve}
Let $\{a(n)\}_{n\geq 1}$ be a sequence of complex numbers such that $|a(n)|\leq 1$ for all $n$. Let $x$ be large and $10\leq y\leq z$ be real numbers. 
Then for all positive integers $k$ such that $ k\leq \log x/(10\log z)$ we have
\begin{align*}
\frac{1}{|\Df|}&\sum_{d\in \Df}\Big|\sum_{y\leq n\leq z}\frac{a(n)\Lambda(n)\chi_d(n)}{\sqrt{n}}\Big|^{2k}\\
&\ll \left(20k\sum_{y\leq p\leq z}\frac{|a(p)|^2(\log p)^2}{p}\right)^k + \left(3\sum_{\sqrt{y}\leq p\leq \sqrt{z}}\frac{|a(p^2)|\log p}{p}\right)^{2k}+ \left(c_0 y^{-1/3}\right)^{k},
\end{align*}
for some positive constant $c_0.$ 

Moreover, the same bound holds for $\displaystyle{\ex\left(\Big|\sum_{y\leq n\leq z}\frac{a(n)\Lambda(n)\X(n)}{\sqrt{n}}\Big|^{2k}\right)}$, for all integers $k\geq 1.$
\end{lem}
\begin{proof}
    We shall only prove the bound for the sum over $d$, since the proof of the corresponding bound for the random model is similar and simpler. First, we have 
$$
\sum_{y\leq n\leq z}\frac{a(n)\Lambda(n)\chi_d(n)}{\sqrt{n}}= \sum_{y\leq p\leq z}\frac{a(p)(\log p)\chi_d(p)}{\sqrt{p}}+\sum_{\substack{\sqrt{y}\leq p\leq \sqrt{z}\\ p\nmid d}}\frac{a(p^2)\log p}{p}+O\left( y^{-1/6}\right), 
    $$
    since the contribution of prime powers $p^k$ with $k\geq 3$ is 
$$ \ll \sum_{k\geq 3} \sum_{p^k\geq y} \frac{\log p}{p^{k/2}}\ll y^{-1/6}. $$
Now, using the basic inequality $|a+b+c|^{k}\leq 3^k (|a|^k+|b|^k+|c|^k)$ (which is valid for all real numbers $a, b, c$ and positive integers $k$), we obtain 
\begin{align*}
&\frac{1}{|\Df|}\sum_{d\in \Df}\Big|\sum_{y\leq n\leq z}\frac{a(n)\Lambda(n)\chi_d(n)}{\sqrt{n}}\Big|^{2k}\\
&\ll \frac{9^k}{|\Df|}\sum_{d\in \Df}\Big|\sum_{y\leq p\leq z}\frac{a(p)(\log p) \chi_d(p)}{\sqrt{p}}\Big|^{2k} + \left(3\sum_{\sqrt{y}\leq p\leq \sqrt{z}}\frac{|a(p^2)|\log p}{p}\right)^{2k}+ \left(c_0 y^{-1/3}\right)^{k},
\end{align*}
for some positive constant $c_0.$
Furthermore, we have 
\begin{equation}\label{Eq.ExpandMoments}
\begin{aligned}
 &\sum_{d\in \Df}\Big|\sum_{y\leq p\leq z}\frac{a(p)(\log p) \chi_d(p)}{\sqrt{p}}\Big|^{2k}\\
 &= \sum_{d\in \Df}\sum_{y\leq p_1,\dots,p_{2k}\leq z}\frac{a(p_1) \cdots a(p_{k})\overline{a(p_{k+1})} \cdots \overline{a(p_{2k}})(\log  p_1)\cdots (\log p_{2k})\chi_d(p_1\cdots p_{2k})}{(p_1p_2\cdots p_{2k})^{1/2}}.
 \end{aligned}
 \end{equation}
The diagonal terms $p_1\cdots p_{2k}=\square$ contribute
$$ \ll x\frac{(2k)!}{2^k k!}\left(\sum_{y\leq p\leq z}\frac{|a(p)|^2(\log p)^2}{p}\right)^k\leq x\left(2k\sum_{y\leq p\leq z}\frac{|a(p)|^2(\log p)^2}{p}\right)^k.$$
On the other hand, if $p_1p_2...p_{2k}\neq \square$ and $p_i\leq z$ then Lemma \ref{Lem.Orthogonality} gives
$$
\sum_{d\in \Df}\chi_d(p_1p_2...p_{2k})\ll x^{4/5},
$$
since $p_1p_2\cdots p_{2k}\leq z^{2k}\leq x$. This  implies that the contribution of these terms to \eqref{Eq.ExpandMoments} is
$$ \ll x^{4/5}\left(\sum_{y\leq p\leq z}\frac{\log p}{\sqrt{p}}\right)^{2k}\ll x^{19/20},$$
by the prime number theorem, and using our assumption on $z$.
Combining the above estimates completes the proof.
\end{proof}



\section{Approximating $-\frac{L'}{L}(s, \chi_d)$ by short Dirichlet polynomials}\label{Section:ApproxLogDerL}
To shorten our notation, for the rest of this paper, we define 
$$ \Ld(s):=-\frac{L'}{L}(s, \chi_d).$$
The goal of this section is to approximate $\Ld(s)$ by short Dirichlet polynomials, if $s$ is slightly to the right of $1/2$. In order to do that, we will use ideas of Selberg from \cite{Sel46} and \cite{SelUnp}. For $d\in \mathcal{D}(x)$ and $2\leq y\leq x$, we let
\begin{align}\label{Def. sigma_y, d}
    \sigma_{y, \: d}:= \dfrac{1}{2} + 2\max_{\mathcal{G}_{y, d}}\bigg(\beta -\dfrac{1}{2}, \dfrac{2}{\log y}\bigg),
\end{align}
where
\begin{align*}
    \mathcal{G}_{y, \: d}: =\{\rho=\beta +i\gamma\colon L(\rho, \chi_{d})= 0, |\gamma-t|\leq y^{3(\beta-1/2)}/\log y\}.
\end{align*}
Next, for $2\leq y\leq x$, we set
\begin{equation}
\label{Eq.DefLambdaYD}
    \Lambda_{y, \: d}(n):=\Lambda(n)\chi_d(n)w_y(n),
\end{equation}
where
 \[ \omega_y(n)=
 \begin{cases} 1 & \text{ if } \ \ n\leq y,\\
 \frac{\log^2(y^3/n)- 2\log^2(y^2/n)}{2\log^2y} & \text{ if } \ \ y\leq n\leq y^2,\\
 \frac{\log^2(y^3/n)}{2\log^2 y} & \text{ if } \ \ y^2\leq n\leq y^3,\\
 0 & \text{ if} \ \ n>y^3.
 \end{cases}
 \]
 Note that $0\leq w_y(n)\leq 1$ for all $n$. We shall use the following lemma due to Selberg  \cite{Sel46}.  
\begin{lem}\label{Lemma: analogue of Selberg}
Let $d\in \Df$ and $10\leq y\leq x$. We have 
\begin{equation}\label{Eq.SelbergZerosPrimes}
   \sum_{\rho} \frac{\sigma_{y, d}-\frac12}{|\sigma_{y, d}-\rho|^2}\ll \log d+ \bigg|\sum_{n\leq y^3}\dfrac{\Lambda_{y,d}(n)}{n^{\sigma_{y, d}}}\bigg|,  
 \end{equation}
 where the sum runs over the non-trivial zeros of $L(s, \chi_d)$. Moreover, 
 \begin{equation}
\label{Eq.ApproxLogLDirichlet}\log L(\sigma_{y, d}, \chi_d)=\sum_{n\leq y^3}\dfrac{\Lambda_{y,\: d}(n)}{n^{\sigma_{y, d}} \log n} + O\bigg(\frac{1}{\log y}\bigg|\sum_{n\leq y^3}\dfrac{\Lambda_{y,\: d}(n)}{n^{\sigma_{y,d}}}\bigg| + \frac{\log x}{\log y}\bigg),
\end{equation}
and for $s=\sigma+it$ with $\sigma\geq \sigma_{y, \: d}$ and $|t|\leq 1$, we have
\begin{equation}
\label{Eq.ApproxL_dDirichlet}\mc{L}_d(s)=\sum_{n\leq y^3}\dfrac{\Lambda_{y,\: d}(n)}{n^s} + O\bigg(y^{(1/2-\sigma)/2}\bigg|\sum_{n\leq y^3}\dfrac{\Lambda_{y,\: d}(n)}{n^{\sigma_{y,d}+it}}\bigg| + y^{(1/2-\sigma)/2}\log d\bigg).
\end{equation}

    \end{lem}
    \begin{proof}
        Selberg proved these estimates for the Riemann zeta function in pages 22-26 of \cite{Sel46}. The analogous estimates for Dirichlet $L$-functions hold 
mutatis mutandis (see Lemma 2.6 of \cite{Ho14}). 
    \end{proof}
We now record the following zero density estimates for the family $\{L(s, \chi_d)\}_{d\in \Df}$ near the critical line, which follows from the work of Conrey and Soundararajan \cite{CoSo2002}.
\begin{lem}[Theorem 2.7 of \cite{Ho14}]\label{Lemma: zero density}
        Let $x$ be large and $\delta>0$ be a small positive constant. There exists $\theta=\theta(\delta)>0$ such that uniformly in $1/2+4/\log x<\sigma<1$ and $10/\log x<T<x^\delta$ we have
        \begin{align*}
            \dfrac{1}{|\mc{D}(x)|}\sum_{d\in \mc{D}(x)}\#\bigg\{\rho=\beta + i\gamma\colon L(\rho, \chi_{d})=0, \beta>\sigma, |\gamma|\leq T\bigg\}\ll x^{-\theta (\sigma-1/2)}T\log x.
        \end{align*}
    \end{lem}
Using this result we show that for almost all $d\in \D(x)$ we have $\sigma_{y, d}=1/2+4/\log y$ if $\log x/\log y\to \infty$.  This will allow us to conclude that for complex numbers $z$ in the range $1/2+4/\log y\leq \re(z)\leq 1$ and $|\im(z)|\leq 1$, the approximation \eqref{Eq.ApproxL_dDirichlet} holds for almost all $d\in \mc{D}(x)$.

    \begin{lem}\label{Lemma: exceptional d}
       Let $x$ be large and $10\leq y\leq x$ be such that $\log x/\log y \to \infty$ as $x\to \infty$.  Define
       \begin{align*}
           \mc{D}_y(x):=\{d\in \mc{D}(x)\colon
\sigma_{y, d} = 1/2+4/\log y\}.
\end{align*}
Then, there exists a constant $C_0>0$ such that
    \begin{align*}
           \big| \D(x)\setminus \D_y(x)\big|\ll x\exp\left(-C_0\frac{\log x}{\log y}\right).
       \end{align*}
    \end{lem}

    \begin{proof}
        Let $\sigma=1/2+4/\log y$. By the definition of $\sigma_{y, \: d}$, if for $d\in \mc{D}(x)$ we have $\sigma_{y, d}>\sigma$, then there exists $\rho_0=\beta_0 + i\gamma_0$ such that $L(\rho_0, \chi_{d})=0$, 
        \[\beta_0> \dfrac{1}{2}+ \dfrac{2}{\log y},  \quad \text{and} \quad |\gamma_0|\leq \dfrac{y^{3(\beta_0-1/2)}}{\log y}.\]
        Write $\sigma^\prime:= 1/2+2/\log y$. Then, we have
        \begin{align*}
         &\dfrac{1}{|\mc{D}(x)|} \#\{d\in \mc{D}(x)\colon \sigma_{y,  d}>\sigma\}\\
         &\ll \dfrac{1}{|\mc{D}(x)|}\sum_{d\in \mc{D}(x)}\#\{\exists \: \rho=\beta + i\gamma\colon L(\rho, \chi_{d})=0, \ \beta>\sigma^\prime, \ |\gamma|\leq 2y^{3(\beta-1/2)}/\log y\} \\
         &\ll \dfrac{1}{|\mc{D}(x)|}\sum_{d\in \mc{D}(x)}\sum_{j=2}^{\log y}\#\{\exists \: \rho=\beta+i\gamma\colon L(\rho, \chi_{d})=0, \ \beta-1/2>j/\log y, \ |\gamma|\leq 2e^{3(j+1)}/\log y\}.
        \end{align*}
        Applying Lemma \ref{Lemma: zero density}, we see that the above quantity is
        \begin{align*}
            \ll \sum_{j=4}^{\log y}x^{-\theta j/\log y}e^{3(j+1)}\frac{\log x}{\log y}\ll \dfrac{\log x}{\log y}e^{-\theta\log x/\log y}\ll e^{-\frac{\theta}{2}\log x/\log y},
        \end{align*}
        as desired.
    \end{proof}
    
For a complex number $z$ with $\re(z)>1/2$, we define $$V_z:= \frac{1}{\re(z)-1/2}.$$
We also set 
\begin{align*}
    \mc{L}_{\text{rand}}(z):=\sum_{n=1}^{\infty}\dfrac{\Lambda(n)\X(n)}{n^z}.
\end{align*}
Note that this series converges almost surely in the half plane $\re(z)>1/2$ by Kolmogorov's three series theorem. We end this section by proving upper bounds for the moments of $\Ld(z)$ and $\Lrand(z)$ when $(\re(z)-1/2)\log x\to \infty$ and $\im(z)$ is bounded.

\begin{lem}\label{Lemma: Moment Bound}
Let $x$ be large and $\nu(x)\to \infty$ as $x\to \infty$. Let $z$ be  a complex number such that $1/2+\nu(x)/\log x\leq \re(z)\leq 1$ and $|\im(z)|\leq 1$. Let $y=\exp\big(10V_z \log(\log x/V_z)\big)$, and  $k\leq (\log x)/(30\log y)$ be a positive integer.  Define 
\begin{align*}
\mc{D}_{z}(x):=\{d\in \mc{D}(x)\colon \sigma_{y, d} = 1/2+4/\log y\}.
\end{align*}
    Then, there exist constants $C_1, C_2>0$ such that 
        \begin{align*}
            \sum_{d\in \mc{D}_{z}(x)}|\mc{L}_d(z)|^{2k}\ll x(C_1kV_z^2)^k \quad \text{and} \quad \ex(|\mc{L}_{\text{rand}}(z)|^{2k})\ll (C_2kV_z^2)^k.
        \end{align*}
    \end{lem}
    \begin{proof}
    We will only establish the desired bound for the $2k$-th moment of $\Ld(z)$, since the corresponding bound for the random model follows along the same lines.  If $d\in \mc{D}_{z}(x)$ and $\sigma_{y, d}\leq \re(z)\leq 1$, then by Lemma \ref{Lemma: analogue of Selberg} we have
        \begin{align*}
\mc{L}_d(z)=\sum_{n\leq y^3}\dfrac{\Lambda_{y, d}(n)}{n^z} + O\bigg(y^{-1/(2V_z)}\bigg|\sum_{n\leq y^3}\dfrac{\Lambda_{y, d}(n)}{n^{\sigma_{y, d}+i t}}\bigg| + y^{-1/(2V_z)}\log d\bigg),
        \end{align*}
     where $t=\im(z)$.   Therefore, using the basic inequality $|a+b+c|^{2k}\leq 3^{2k}(|a|^{2k}+|b|^{2k}+|c|^{2k})$  we infer from Lemma \ref{Lem.LargeSieve} that
     \begin{equation}\label{Eq.BoundMomentsLdSpecial}
\begin{aligned}
    \sum_{d\in \mc{D}_{z}(x)}|\mc{L}_d(z)|^{2k}
    &\ll 9^k\sum_{d\in \mc{D}_{z}(x)}\bigg|\sum_{n\leq y^3}\dfrac{\Lambda_{y, d}(n)}{n^z}\bigg|^{2k} + 9^k y^{-k/V_z}\sum_{d\in \mc{D}_{z}(x)} \bigg|\sum_{n\leq y^3}\dfrac{\Lambda_{y, d}(n)}{n^{\sigma_{y, d}+it}}\bigg|^{2k}\\
    &\quad \quad \quad + 9^k x y^{-k/V_z}(\log x)^{2k}\\
    &\ll x\bigg(200k\sum_{p\leq y^3}\dfrac{(\log p)^2}{p^{2\re(z)}}\bigg)^k + x\bigg(30\sum_{p\leq y^{3/2}}\dfrac{\log p}{p^{2\re(z)}}\bigg)^{2k}+
9^k x y^{-k/V_z}(\log x)^{2k}\\
            & \quad \quad \quad +  x y^{-k/V_z} \left(200k\sum_{p\leq y^3}\frac{(\log p)^2}{p} \right)^k + x y^{-k/V_z}\bigg(30\sum_{p\leq y^{3/2}}\dfrac{\log p}{p}\bigg)^{2k}\\
            &\ll x(C_1kV_z^2)^k,
\end{aligned}
\end{equation}
        for some positive
        constant $C_1$, by our assumptions on  $z$ and $k$, and since 
    \begin{equation}\label{Eq.EstimatesSumsPrimes}
    \sum_{p}\frac{(\log p)^2}{p^{2\re(z)}}\asymp V_z^2 \ \text{ and } \ \sum_{p}\frac{(\log p}{p^{2\re(z)}} \asymp V_z,
    \end{equation}
    by partial summation and the prime number theorem.
    \end{proof}

\section{A discrepancy bound for the distribution of $-\frac{L'}{L}(s, \chi_d)$}\label{Section:Discrepancy}
Throughout this section we let $\nu$ be a positive function such that $\nu(x)\to\infty$ as $x\to \infty$. Let $x$ be large and $z$ be a real number such that $1/2+\nu(x)/\log x\leq z\leq 1$. Put $y=\exp\big(20V_z \log(\log x/V_z)\big)$, and  define 
\begin{align*}
\mc{D}_{z}(x):=\{d\in \mc{D}(x)\colon \sigma_{y, d} = 1/2+4/\log y\}.
\end{align*} 
Then $|\D_z(x)|\sim |\D(x)|$ by Lemma \ref{Lemma: exceptional d}. Moreover, for any real number $u$, we define
$$ \Phi_{x, z}(u):= \frac{1}{|\mc{D}_z(x)|}\sum_{d\in \mc{D}_z(x)} \exp\left(2\pi i  u \frac{\Ld(z)}{V_z}\right),$$
and 
$$ \Phi_{\textup{rand}, z}(u)= \ex\left[ \exp\left(2\pi i  u \frac{\Lrand(z)}{V_z}\right)\right].
$$
Furthermore, we define the ``discrepancy'' between the distribution functions of $\Ld(z)/V_z$ and $\Lrand(z)/V_z$ as 
\begin{align*}
 D(z):=\sup_{t\in \mathbb{R}}\bigg|\dfrac{1}{|\mc{D}_z(x)|}|\{d\in \mc{D}_z(x)\colon \mc{L}_d(z)/V_z\leq t\}|-\mb{P}\big(\mc{L}_{\text{rand}}(z)/V_z \leq t\big)\bigg|.
\end{align*}
The goal of this section is to prove the following theorem
\begin{thm}\label{Theorem: Discrepancy}

Let $1/2+\nu(x)/\log x\leq z\leq 1$ with $\nu(x)\to \infty$ as $x\to \infty$. Then, we have
$$
    D(z)\ll \bigg(\dfrac{V_z\log (\log x/V_z)}{\log x}\bigg)^{1/2}.
    $$
    \end{thm}

We start by proving the following lemma.
\begin{lem}\label{Lemma: discrepancy}
    Let $x, \nu, z$ and $y$ be as above. Then, for all real numbers $u$ such that $(V_z/\log x)^2\leq |u|\leq (\log x/V_z)^5$, we have
    \begin{align*}
       \Phi_{x, z}(u) = \dfrac{1}{|\mc{D}(x)|}\sum_{d\in \mc{D}(x)}\exp\bigg(2\pi i\frac{u}{V_z}\sum_{n\leq y}\dfrac{\Lambda(n)\chi_{d}(n)}{n^z}\bigg) + O\bigg( |u|\frac{V_z^5}{(\log x)^5}\bigg).
    \end{align*}
\end{lem}
\begin{proof}
     By Lemma \ref{Lemma: analogue of Selberg}, we have
    \begin{align*}
        \sum_{d\in \mc{D}_z(x)} \exp\left(2\pi i  u \frac{\Ld(z)}{V_z}\right)=\sum_{d\in \mc{D}_z(x)}\exp\bigg(2\pi i\dfrac{u}{V_z}\sum_{n\leq y^3}\dfrac{\Lambda_{y, d}(n)}{n^z}\bigg) + E_1,
    \end{align*}
    where
    \begin{align*}
        E_1\ll \dfrac{|u|}{V_z}y^{-1/(2V_z)}\bigg(\sum_{d\in \mc{D}_z(x)}\bigg|\sum_{n\leq y^3}\dfrac{\Lambda_{y, \: d}(n)}{n^{\sigma_{y, d}}}\bigg| + x\log x\bigg).
    \end{align*}
    By the Cauchy-Schwarz inequality and Lemma \ref{Lem.LargeSieve}, we have
\begin{align*}
    \sum_{d\in \mc{D}_z(x)}\bigg|\sum_{n\leq y^3}\dfrac{\Lambda_{y, \: d}(n)}{n^{\sigma_{y, d}}}\bigg|
        &\leq x^{1/2}\bigg(\sum_{d\in \mc{D}(x)}\bigg|\sum_{n\leq y^3}\dfrac{\Lambda_{y, \: d}(n)}{n^{\sigma_{y, d}}}\bigg|^2\bigg)^{1/2}\\
        &\ll x\bigg(\sum_{p\leq y^3}\dfrac{(\log p)^2}{p^{2\sigma_{y, d}}}\bigg)^{1/2} + x \bigg(\sum_{p\leq y^{3/2}}\dfrac{\log p}{p^{2\sigma_{y, d}}}\bigg) \ll x\log x,
    \end{align*}
    since $2\sigma_{y, d}>1$. Hence, we get
    \begin{align*}
         \sum_{d\in \mc{D}_z(x)} \exp\left(2\pi i  u \frac{\Ld(z)}{V_z}\right)=\sum_{d\in \mc{D}_z(x)}\exp\bigg(2\pi i\dfrac{u}{V_z}\sum_{n\leq y^3}\dfrac{\Lambda_{y, d}(n)}{n^z}\bigg) + O\bigg(x|u| \Big(\frac{V_z}{\log x}\Big)^9\bigg),
    \end{align*}
since $y^{-1/V_z}= (V_z/\log x)^{20}$. Next, we write
\begin{align*}
    \sum_{d\in \mc{D}_z(x)}\exp\bigg(2\pi i\dfrac{u}{V_z}\sum_{n\leq y^3}\dfrac{\Lambda_{y, d}(n)}{n^z}\bigg)=\sum_{d\in \mc{D}_z(x)}\exp\bigg(2\pi i\dfrac{u}{V_z}\sum_{n\leq y}\dfrac{\Lambda(n)\chi_{d}(n)}{n^z}\bigg) + E_2,
\end{align*}
where
\begin{align*}
    E_2
    & \ll\dfrac{|u|}{V_z}\sum_{d\in \mc{D}_z(x)}\bigg|\sum_{n\leq y^3}\dfrac{\Lambda_{y, d}(n)}{n^z}-\sum_{n\leq y}\dfrac{\Lambda(n)\chi_{d}(n)}{n^z}\bigg|\\
    &\ll \dfrac{|u|}{V_z}\sum_{d\in \mc{D}(x)}\bigg|\sum_{y< n\leq  y^3}\dfrac{\Lambda_{y, d}(n)}{n^z}\bigg|,
\end{align*}
using the definition of $\Lambda_{y, \:d}$. Applying the Cauchy-Schwarz inequality and Lemma \ref{Lem.LargeSieve}  we obtain
\begin{align*}
\sum_{d\in \mc{D}(x)}\bigg|\sum_{y<n\leq y^3}\dfrac{\Lambda_{y, d}(n)}{n^z}\bigg|& \leq x^{1/2}\bigg(\sum_{d\in \mc{D}(x)}\bigg|\sum_{y<n\leq y^3}\dfrac{\Lambda_{y, \: d}(n)}{n^z}\bigg|^2\bigg)^{1/2}\\
        &\ll x\bigg(\sum_{y<p\leq y^3}\dfrac{(\log p)^2}{p^{2z}}\bigg)^{1/2} + x \bigg(\sum_{\sqrt{y}< p\leq y^{3/2}}\dfrac{\log p}{p^{2z}}\bigg) +x y^{-1/6} \\
&\ll y^{-(z-1/2)/3}\log x\ll \frac{V_z^6}{(\log x)^5},
\end{align*}
since 
\begin{equation}
\label{Eq.EstimateTailPrimes1}
\sum_{p>y}\dfrac{(\log p)^2}{p^{2z}}\ll \frac{\log y}{y^{2z-1} (z-1/2)}+ \frac{1}{(z-1/2)^2 y^{2z-1}},
\end{equation} 
and  
\begin{equation}\label{Eq.EstimateTailPrimes2}
\sum_{p>\sqrt{y}}\dfrac{\log p}{p^{2z}} \ll \frac{1}{y^{z-1/2}(z-1/2)},
\end{equation}
by partial summation and the prime number theorem. Finally, we note that 
$$ \frac{1}{|\mc{D}_z(x)|}\sum_{d\in \mc{D}_z(x)}\exp\bigg(2\pi i\dfrac{u}{V_z}\sum_{n\leq y}\dfrac{\Lambda(n)\chi_{d}(n)}{n^z}\bigg) = \frac{1}{|\D(x)|}\sum_{d\in \mc{D}(x)}\exp\bigg(2\pi i\dfrac{u}{V_z}\sum_{n\leq y}\dfrac{\Lambda(n)\chi_{d}(n)}{n^z}\bigg) +E_3, $$
where 
$$ E_3\ll \frac{|\Df\setminus\D_z(x)|}{|\D(x)|}\ll \exp\Big(-\theta \frac{\log x}{\log y}\Big),$$
by Lemma \ref{Lemma: exceptional d}.
Collecting the above estimates completes the proof.
\end{proof}

\begin{pro}\label{Prop: Characteristic}
Let $x$, $\nu$, $z$ and $y$ be as above. There exists a constant $c_1>0$ such that for all real numbers $u$ with $(V_z/\log x)^2\leq |u|\leq c_1\sqrt{\log x/(V_z\log (\log x/V_z))}$, we have
    \begin{align*}
       \Phi_{x, z}(u)=\Phi_{\textup{rand}, z}(u) + O\bigg(|u|\frac{V_z^4}{(\log x)^{4}}\bigg).
    \end{align*}
\end{pro}

\begin{proof}

By Lemma \ref{Lemma: discrepancy}, we have
\begin{align*}
   \Phi_{x, z}(u)= \dfrac{1}{|\mc{D}(x)|}\sum_{d\in \mc{D}(x)} \exp\bigg(2\pi i\frac{u}{V_z}\sum_{n\leq y}\dfrac{\Lambda(n)\chi_{d}(n)}{n^z}\bigg) + O\bigg(|u|\frac{V_z^5}{(\log x)^5}\bigg).
\end{align*} 
   Next, we deal with the main term in the above expression. Let $N=\lfloor(\log x)/(50 \log y)\rfloor$. By applying the Taylor expansion of $e^{2\pi it}$ for real $t$, we see that
   \begin{align*}
       \dfrac{1}{|\mc{D}(x)|}&\sum_{d\in \mc{D}(x)} \exp\bigg(2\pi i\frac{u}{V_z}\sum_{n\leq y}\dfrac{\Lambda(n)\chi_{d}(n)}{n^z}\bigg)\\
       &=\sum_{k=0}^{2N-1}\dfrac{(2\pi iu)^k}{V_z^k k!} \dfrac{1}{|\mc{D}(x)|}\sum_{d\in \mc{D}(x)}\bigg(\sum_{n\leq y}\dfrac{\Lambda(n)\chi_{d}(n)}{n^z}\bigg)^k + E_4,
   \end{align*}
   where
   \begin{align*}
       E_4
       &\ll \dfrac{(2\pi u)^{2N}}{V_z^{2N}(2N)!}\dfrac{1}{|\mc{D}(x)|}\sum_{d\in \mc{D}(x)}\bigg(\sum_{n\leq y}\dfrac{\Lambda(n)\chi_{d}(n)}{n^z}\bigg)^{2N}\\
       &\ll \dfrac{(2\pi u)^{2N}}{V_z^{2N}(2N)!} \cdot (c_2NV_z^2)^N\ll (c_3u^2/N)^N\ll e^{-N},
   \end{align*}
  for some positive constants $c_2, c_3$, where the second inequality follows by the same calculations leading to \eqref{Eq.BoundMomentsLdSpecial}, and the third from    Stirling's formula.  Therefore,
   \begin{align}\label{eq: Discrepancy prop 1}
       \Phi_{x, z}(u)=\sum_{k=0}^{2N-1}\dfrac{(2\pi iu)^k}{V_z^k k!} \dfrac{1}{|\mc{D}(x)|}\sum_{d\in \mc{D}(x)}\bigg(\sum_{n\leq y}\dfrac{\Lambda(n)\chi_{d}(n)}{n^z}\bigg)^k + O\bigg(|u|\frac{V_z^5}{(\log x)^5}\bigg).
   \end{align}
    On the other hand, by Lemma \ref{lem.AsympMoments}, we have
\begin{align}\label{eq: Discrepancy prop 2}
     \notag  \bigg|\sum_{k=0}^{2N-1}&\dfrac{(2\pi iu)^k}{V_z^k k!} \bigg(\dfrac{1}{|\mc{D}(x)|}\sum_{d\in \mc{D}(x)}\bigg(\sum_{n\leq y}\dfrac{\Lambda(n)\chi_{d}(n)}{n^z}\bigg)^k-\mathbb{E}\bigg(\sum_{n\leq y}\dfrac{\Lambda(n)\mb{X}(n)}{n^z}\bigg)^k\bigg)\bigg|\\
       &\ll x^{-1/5}\sum_{k=0}^{2N-1}\bigg(\frac{c_4uy}{V_zk}\bigg)^k\ll x^{-1/5}Ny^{2N}\ll x^{-1/10},
   \end{align}
   which is negligible. Here we have used our assumptions on $u$ and $N$ to bound the sum over $k$. 

 We now handle the characteristic function of the random model. Let $\A$ denote the event
 $$ \bigg|\sum_{n>y} \frac{\Lambda(n)\X(n)}{n^z}\bigg|\leq  B:= \frac{V_z^6}{(\log x)^5}.
 $$
 Let $\kappa$ be a positive integer to be chosen. Then, by Markov's inequality and Lemma \ref{Lem.LargeSieve} (letting $z\to \infty$ therein) we obtain 
\begin{align*} \pr(\A^c)&\leq \frac{1}{B^{2\kappa}}\ex\bigg|\sum_{n>y} \frac{\Lambda(n)\X(n)}{n^z}\bigg|^{2\kappa}\\
& \ll \left(c_5 \frac{\kappa}{B^2} \sum_{p>y} \frac{(\log p)^2}{p^{2z}}\right)^{\kappa}+\left(c_6 \sum_{p>\sqrt{y}} \frac{\log p}{p^{2z}}\right)^{2\kappa}
\ll \left(c_7\frac{\kappa V_z\log y}{B^2 y^{2z-1}} \right)^{k},
\end{align*}
for some positive constants $c_5, c_6$ and $c_7$, where the last bound follows from \eqref{Eq.EstimateTailPrimes1} and \eqref{Eq.EstimateTailPrimes2}. Choosing $\kappa=\lfloor B^2y^{2z-1}/(ec_7 V_z\log y)\rfloor$ and using that $y^{2z-1}= (\log x)^{40}/V_z^{40}$ we deduce that
$$ \pr(\A^c) \ll e^{-\kappa}\ll \exp\left(-\frac{\log x}{V_z}\right).$$
Letting $\mathbf{1}_{\A}$ denote the indicator function of the event $\A$, we therefore get
\begin{equation}\label{Eq.ApproxCharModel}
\begin{aligned}\Phi_{\text{rand}, z}(u)&=\ex\left[ 
\mathbf{1}_{\A}\cdot\exp\left(2\pi i  u \frac{\Lrand(z)}{V_z}\right)\right]+ O\left(\exp\left(-\frac{\log x}{V_z}\right)\right)\\
&=\ex\left[ 
\mathbf{1}_{\A}\cdot\exp\left(\frac{2\pi i  u}{V_z}\sum_{n\leq y}\frac{\Lambda(n)\X(n)}{n^z}+ O\left(\frac{|u| V_z^5}{(\log x)^5}\right) \right)\right]+ O\left(\exp\left(-\frac{\log x}{V_z}\right)\right)\\
& = \ex\left[ 
\exp\left(\frac{2\pi i  u}{V_z}\sum_{n\leq y}\frac{\Lambda(n)\X(n)}{n^z} \right)\right]+ O\left(\frac{|u| V_z^5}{(\log x)^5}\right).
\end{aligned}
\end{equation}
Next, by the same argument leading to \eqref{eq: Discrepancy prop 1} together Lemma \ref{Lem.LargeSieve}, we obtain
 \begin{align}\label{eq: Discrepancy prop 3}
     \ex\left[ 
\exp\left(\frac{2\pi i  u}{V_z}\sum_{n\leq y}\frac{\Lambda(n)\X(n)}{n^z} \right)\right]=\sum_{k=0}^{2N-1}\dfrac{(2\pi i u)^k}{V_z^kk!}\mathbb{E}\bigg(\sum_{n\leq y}\dfrac{\Lambda(n)\mb{X}(n)}{n^z}\bigg)^k + O(e^{-N}).
 \end{align}
 Combining \eqref{eq: Discrepancy prop 1}, \eqref{eq: Discrepancy prop 2}, \eqref{Eq.ApproxCharModel} and \eqref{eq: Discrepancy prop 3} completes the proof.
 \end{proof}

 Next, we show that the characteristic function of $\Lrand(z)/V_z$ decays exponentially on $\mathbb{R}$, uniformly in $1/2<z\leq 1$.

\begin{lem}\label{Lemma: decay of characteristic of random model}
Let $1/2<z\leq 1$. Then, there exists an absolute constant $C_0>0$  such that for all $u\in \mathbb{R}$ we have 
\begin{align*}
\Phi_{\textup{rand}, z}(u)\ll \exp\left(-C_0\frac{|u|^{1/z}}{\log(|u|+1)^{2-1/z}}\right).
\end{align*}
\end{lem}
\begin{proof}
Let $A$ be a suitably large constant. Since $|\Phi_{\textup{rand}, z}(u)|\leq 1$ for all real numbers $u$, we may assume that $|u|>A$. First, note that 
$$\Lrand(z)= \sum_{n=1}^{\infty}\frac{\Lambda(n)\X(n)}{n^z}=\sum_{p} \log p\sum_{k=1}^{\infty} \frac{\X(p)^k}{p^{kz}}= \sum_{p} \frac{\X(p)\log p}{p^{z}-\X(p)}, 
$$ 
and hence 
$$ \Phi_{\text{rand}, z}(u)
    =\prod_{p>2} \ex\bigg[\exp\bigg(2\pi i u\frac{\X(p)\log p}{V_z(p^z-\X(p))}\bigg)\bigg],$$
    since $\{\X(p)\}_{p\: \text{prime}}$ are independent and $\X(2)=0.$
Now for any odd prime $p$, by Taylor's expansion, we have 
\begin{align*}
&\exp\bigg(2\pi i u\frac{\X(p)\log p}{V_z(p^z-\X(p))}\bigg) \\
&=  1+ 2\pi i u\frac{\X(p)\log p}{V_z(p^z-\X(p))}- 2\pi^2 u^2 \frac{\X(p)^2(\log p)^2}{V_z^2(p^z-\X(p))^2}+ O\left(|u|^3\frac{(\log p)^3}{V_z^3p^{3z}}\right) \\
& = 1+ 2\pi i u\frac{\X(p)\log p}{V_z p^z} +2\pi i u \frac{\X(p)^2\log p}{V_z p^{2z}} - 2\pi^2 u^2 \frac{\X(p)^2(\log p)^2}{V_z^2 p^{2z}} \\
& \quad + O\left(|u|\frac{\log p}{V_z p^{3z}}+|u|^3\frac{(\log p)^3}{V_z^3p^{3z}}\right).
\end{align*}
Since $\ex (\X(p))=0$ and $\ex(\X(p)^2)=1-1/(p+1)$ we get
$$\ex\left[\exp\bigg(2\pi i u\frac{\X(p)\log p}{V_z(p^z-\X(p))}\bigg)\right]
 =  1 - 2\pi^2 u^2 \frac{(\log p)^2}{V_z^2 p^{2z}} + O\left(|u|\frac{\log p}{V_z p^{2z}}
+|u|^3\frac{(\log p)^3}{V_z^3p^{3z}}\right).
$$
Let 
$$ U= \max\left(e^{AV_z}, (A|u|\log |u|)^{1/z}\right).$$
Then we have
\begin{equation}\label{Eq.ExpansionExpectation}
\begin{aligned}
|\Phi_{\text{rand}, z}(u)|  &\leq \prod_{p\geq U}\left|\ex\bigg[\exp\bigg(2\pi i u\frac{\X(p)\log p}{V_z(p^z-\X(p))}\bigg)\bigg]\right|
\\
& \leq 
\exp\left(-2\pi^2 \frac{u^2}{V_z^2}\sum_{p>U}\dfrac{(\log p)^2}{p^{2z}} + O\bigg(\frac{|u|}{V_z}\sum_{p>U}\frac{\log p}{p^{2z}}+\dfrac{|u|^3}{V_z^3}\sum_{p>U} \frac{(\log p)^3}{p^{3z}}\bigg)\bigg)\right).
\end{aligned}
\end{equation}
Since $U\geq e^{AV_z}$ (and $A$ is suitably large) it follows by partial summation and the prime number theorem  that 
$$ \sum_{p>U} \frac{(\log p)^2}{p^{2z}} \asymp \frac{V_z \log U}{U^{2z-1}}, \ \ \sum_{p>U} \frac{\log p}{p^{2z}}\asymp \frac{V_z}{U^{2z-1}}, \ \text{ and } \sum_{p>U} \frac{(\log p)^3}{p^{3z}} \asymp \frac{(\log U)^2}{U^{3z-1}}. $$ Inserting these estimates in \eqref{Eq.ExpansionExpectation} implies that 
\begin{align*}
|\Phi_{\text{rand}, z}(u)| &\ll \exp\left(-C_1  \frac{u^2\log U}{V_z U^{2z-1}}\left(1+ O\left(\frac{V_z}{|u|\log U}+ \frac{|u|\log U}{V_z^2 U^z}\right)\right)\right)\\
& \ll \exp\left(-\frac{C_1}{2}  \frac{u^2\log U}{V_z U^{2z-1}}\right)\ll \exp\left(-\frac{C_1}{2}  \frac{u^2}{U^{2z-1}}\right),
\end{align*}
for some positive constant $C_1$, by our choice of $U$.  The result follows upon noting that $U^{2z-1}\asymp_A 1$ if $U=e^{AV_z}$, and $U^{2z-1}\asymp_A (|u|\log |u|)^{2-1/z}$ otherwise.   
\end{proof}
It follows from Lemma \ref{Lemma: decay of characteristic of random model} that uniformly in $1/2<z\leq 1$ we have 
$$ \Phi_{\textup{rand}, z}(u) \ll \exp\left(-C_0\frac{|u|}{\log|u|}\right)$$ 
for all $u\in \mathbb{R}$. Thus, by Fourier inversion, the random variable $\Lrand(z)/V_z$ is  absolutely continuous, and has a uniformly bounded density function. In particular, for any $\ep>0$ we have
\begin{equation}
\label{Eq.BoundSmallRandom}\pr\left(\Lrand(z)/V_z\in [-\ep, \ep]\right)\ll \ep,
\end{equation}
where the implied constant is absolute. We are now ready to prove Theorem \ref{Theorem: Discrepancy}.

\begin{proof}[Proof of Theorem \ref{Theorem: Discrepancy}]
Let
  \begin{align*}
      T(z):=c_1\sqrt{\dfrac{\log x}{V_z\log (\log x/V_z)}},
  \end{align*}
  where $c_1$ in the constant in the statement of Proposition \ref{Prop: Characteristic}.
Since $\Lrand(z)/V_z$ has a uniformly bounded density function, it follows from  the Berry-Esseen Theorem (see Theorem 7.16 of \cite{Ten15}) that
  \begin{align*}
      D(z) 
      \ll \dfrac{1}{T(z)} + \int_{-T(z)}^{T(z)}\dfrac{|\Phi_{x, z}(u)-\Phi_{\text{rand}, z}(u)|}{u}\: du.
  \end{align*}
  Note that if $|u|\leq 1/T(z)$,  then by Taylor's expansion and the Cauchy-Schwarz inequality, we have
  \begin{align*}
      \Phi_{x, z}(u)-\Phi_{\text{rand}, z}(u) &=\dfrac{1}{|\mc{D}_z(x)|}\sum_{d\in \mc{D}_z(x)}\exp\bigg(2\pi iu\dfrac{\mc{L}_d(z)}{V_z}\bigg)-\ex\bigg[\exp\bigg(2\pi iu\dfrac{\mc{L}_{\textup{rand}}(z)}{V_z}\bigg)\bigg]\\
      &\ll |u|\bigg(\dfrac{1}{|\mc{D}_z(x)|}\sum_{d\in \mc{D}_z(x)}\dfrac{|\mc{L}_d(z)|}{V_z} + \ex\bigg(\dfrac{|\mc{L}_{\text{rand}}(z)|}{V_z}\bigg)\bigg)\\
      &\leq |u|\bigg(\bigg(\dfrac{1}{|\mc{D}_z(x)|}\sum_{d\in \mc{D}_z(x)}\dfrac{|\mc{L}_d(z)|^2}{V_z^2}\bigg)^{1/2} + |u|\ex\bigg(\dfrac{|\mc{L}_{\text{rand}}(z)|^2}{V_z^2}\bigg)^{1/2}\bigg)\\
      &\ll |u|
  \end{align*}
  by Lemma \ref{Lemma: Moment Bound}. Therefore, we obtain
  \begin{align*}
      D(z)\ll \dfrac{1}{T(z)} + \int_{1/T(z)\leq |u|\leq T(z)}\dfrac{|\Phi_{x, z}(u)-\Phi_{\text{rand}, z}(u)|}{u}\: du.
  \end{align*}
  By invoking Proposition \ref{Prop: Characteristic}, we infer that
  \begin{align*}
      \int_{1/T(z)\leq |u|\leq T(z)}\dfrac{|\Phi_{x, z}(u)-\Phi_{\text{rand}, z}(u)|}{u}\: du
      \ll T(z) \frac{V_z^4}{(\log x)^4}\ll \dfrac{1}{T(z)},
  \end{align*}
which completes the proof.
\end{proof}


\section{Real zeros of $L'(s, \chi_d)$ away from the central point}\label{Section. Thm main away from the central point}
 
The goal of this section is to establish the following result, which is our first step towards proving Theorem \ref{Thm:Main}.

\begin{thm}\label{Thm:RealZerosAway1/2} Let $\nu$ be a positive function such that $\nu(x)\leq \log\log  x$ for large $x$, and $\nu(x)\to \infty$ as $x\to \infty$. 
For all $d\in \Df$ except for a set of cardinality $\ll x\sqrt{\log \nu(x)}/\sqrt{\nu(x)}$, we have
$$ R_d\left(\frac12+ \frac{\nu(x)}{\log x}, 1\right)\ll (\log\log x)(\log\log\log x).$$
\end{thm}

\begin{proof} For $1\leq j\leq J:=\lfloor\frac{1}{\log 3} (\log\log x-\log \nu(x))\rfloor$, we define 
$$z_j:=\frac{1}{2}+\frac{1}{3^j}, \quad r_j:= \frac{1}{2\cdot 3^j}, \quad  \text{ and } R_j:=\frac54 r_j.$$

\begin{figure}[h]
  \centering

\begin{tikzpicture}[>=stealth, scale=24]
   \draw[->] (0.49,0) -- (1.08,0) node[right] {};

    \draw[black] (0.5,-0.2) -- (0.5,0.2) node[above] {$\frac{1}{2}$};
    \draw[black] (1,-0.2) -- (1,0.2) node[above] {1};

    \draw[black, thick, dashed] (0.512,-0.2) -- (0.512,0.2)
        node[above] {$t_x$};

    \draw[blue] (5/6,0) circle (1/6);
    \draw[blue, dashed] (5/6,0) circle (5/24);


    \filldraw[black] (5/6,0) circle (0.05pt) node[below] {$z_1$};
    \filldraw[black] (11/18,0) circle (0.05pt) node[below] {$z_2$};
    
     \filldraw[black] (29/54,0) circle (0.05pt) node[below] {$z_3$};

    \foreach \j [evaluate=\j as \n using int(\j)] in {1, 2, 3} {
        \pgfmathsetmacro{\centerx}{0.5 + 1/(3^\j)}
        \pgfmathsetmacro{\radiusa}{1/(2*3^\j)}
        \pgfmathsetmacro{\radiusb}{5/(8*3^\j)}

        \draw[blue] (\centerx,0) circle (\radiusa);
        \draw[red, dashed] (\centerx,0) circle (\radiusb);


            \path (\centerx,0) ++(129:{\radiusa+0.0001}) node[blue, font=\footnotesize, anchor=north] {$\mathcal{C}_{\n}$};
        \path (\centerx,0) ++(120:{\radiusb + 0.001}) node[red, font=\footnotesize, anchor=south] {$\widetilde{\mathcal{C}}_{\n}$};
    }

\end{tikzpicture}

\caption{Circles covering $[t_x, 1]$, where $t_x=1/2+\nu(x)/\log x$.}\label{Figure cicle 1}
\end{figure}
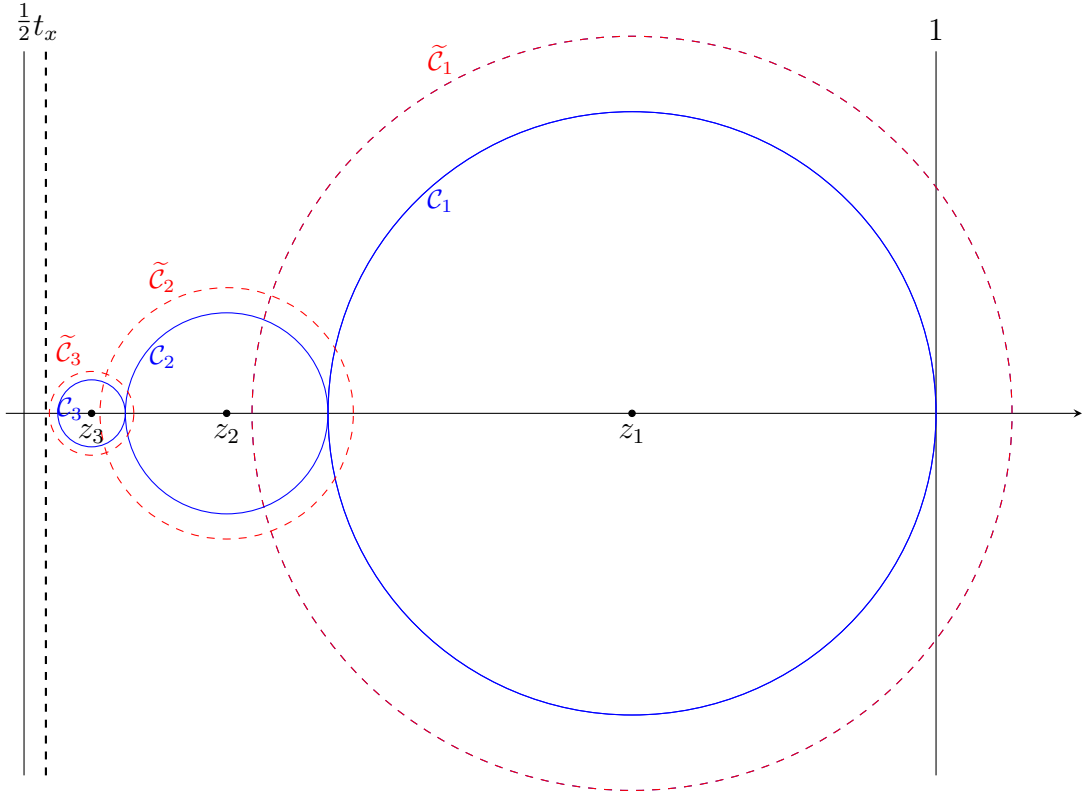
We also let $\mathcal{C}_j$ and $\widetilde{\mathcal{C}}_j$ be the concentric circles of center $z_j$ and radii $r_j$ and $R_j$, respectively (see Figure \ref{Figure cicle 1}). One can observe that 
$$ \mathcal{I}:=\left[ \frac12+ \frac{\nu(x)}{\log x}, 1\right] \subset \bigcup_{j=1}^J \{ z\in \mathbb{C} : |z-z_j|\leq r_j\}. $$
Let $\Ds\subset \D(x)$ be the set of fundamental discriminants such that $L(s, \chi_d)$ has no zeros in the  discs $|z-z_j|\leq \frac{7}{4}r_j$ for all $j\leq J$. Since for each $j\leq J$, such a disc  is contained in the square $\{z: z_j- 7r_j/4\leq \re(z)\leq z_j+ 7r_j/4 \text{ and } |\im(z)|\leq 7r_j/4\}$, it follows from Lemma \ref{Lemma: zero density} that for some absolute positive constant $c_8$ we have  
\begin{equation}\label{Eq.Exceptionala1}
 |\Df\setminus \Ds|\ll x\log x\sum_{j=1}^J \frac{x^{-c_8/3^j}}{3^j} \ll x\sum_{j=1}^J  x^{-c_8/(2\cdot 3^j)}\ll x\exp(-c_{9} \nu(x)),
 \end{equation}
for some positive constant $c_{9}$, since $(\log x)3^{-j}\ll \exp(\frac {c_9}{2}(\log x)3^{-j}), $ for all $j\le J.$

Let $d\in \Ds$. Recall that $\mathcal{L}_d(s)=\frac{-L^\prime}{L}(s, \chi_d)$. Then $\Ld$ is analytic on the open disc $|z-z_j|< 7r_j/4$ for all $j\leq J$, and moreover the number of zeros of $\Ld(s)$ in $\mathcal{I}$ is bounded by 
\begin{equation}\label{Eq.SumZerosDiscs}
\sum_{j=1}^J N_j(\Ld),
\end{equation}
where $N_j(\Ld)$ is the number of zeros of $\Ld(s)$ inside the circle $\mathcal{C}_j$. Since $\Ld$ is analytic inside $\widetilde{\mathcal{C}}_j$, it follows from Jensen's formula that 
\begin{equation}\label{Eq.Jensen}
N_j(\Ld) \leq \frac{\log\big(M_{j, d}/\Ld(z_j)\big)}{\log (R_j/r_j)}=\frac{1}{\log (5/4)}\Big(\log\big(M_{j, d}/V_j\big)- \log\big(|\Ld(z_j)|/V_j\big)\Big),
\end{equation}
where 
$$ M_{j, d}:= \max_{s\in \widetilde{\mathcal{C}}_j} |\Ld(s)|, \textup{ and } V_j:=\frac{1}{z_j-1/2}= 3^j.$$
Note that we  normalized both $M_{j, d}$ and $\Ld(z_j)$ by ``the standard deviation'' $V_j$.
Therefore, in order to bound the sum on \eqref{Eq.SumZerosDiscs} we would like to show that for almost all fundamental discriminants $d\in \Ds$ we have 

\begin{itemize}

\item[1.] $\max_{j\leq J}  M_{j, d}/V_j$ is not too large (namely $\ll (\log\log x)^2$ say).

\smallskip

\item[2.] 
$\min_{j\leq J} |\Ld(z_j)|/V_j$ is not too small (namely $\gg (\log\log x)^{-2}$ say).

\smallskip

\end{itemize}

 We start by handling the first condition. Let $1\leq j\leq J$.  Since $\Ld(s)$ is analytic on the open disc of center $z_j$ and radius $\frac75 R_j$ for all $d\in \Ds$, it follows from Cauchy's formula that
 $$\Ld(s)^{2}= \frac{1}{2\pi i} \int_{|z-z_j|=\frac76R_j} \frac{\Ld(z)^{2}}{z-s} dz, $$
 for all $s\in \widetilde{\mc{C}}_j.$ This implies 
\begin{equation}\label{Eq.Cauchy}
 M_{j, d}^{2}=\max_{s\in \widetilde{\mc{C}}_j} |\Ld(s)|^{2} \ll V_j \int_{|z-z_j|=\frac76R_j} |\Ld(z)|^{2} |dz|, 
 \end{equation}
 since $|z-s|\geq |z-z_j|-|s-z_j| = R_j/6 \asymp 1/V_j$. 
 Let $L$ be a positive parameter to be chosen, and define $\mc{E}_1(x)$ to be the set of fundamental discriminants $d\in \Ds$ such that $\max_{j\leq J} M_{j, d}/V_j\geq L$. The proportion of $d\in \mc{E}_1(x)$  is 
\begin{equation}\label{Eq.BoundMaxMoments}
\begin{aligned}
 \leq \sum_{j=1}^J\frac{1}{(L V_j)^{2}}\frac{1}{|\Ds|}\sum_{d\in \Ds} M_{j, d}^{2}
 &\ll \sum_{j=1}^J\frac{1}{L^{2} V_j}\int_{|z-z_j|=\frac76R_j} \Bigg(\frac{1}{|\Ds|}\sum_{d\in\Ds}|\Ld(z)|^{2}\Bigg) |dz|,\\
 & \ll \sum_{j=1}^J\frac{1}{L^{2} V_j}\int_{|z-z_j|=\frac76 R_j}V_z^2  |dz|
 \end{aligned}
 \end{equation}
 by \eqref{Eq.Cauchy}, Lemma \ref{Lemma: Moment Bound} and the fact that $|\Ds|\asymp x$. Furthermore, since $\int_{|z-z_j|=\frac76R_j}|dz|\asymp 1/V_j$ and $V_z\leq 4V_j$  for all complex numbers $z$ with $|z-z_j|=\frac76 R_j$ (since $\re(z)\geq z_j- \frac{7}{6} R_j\geq \frac 12+ \frac{1}{4V_j}$), we deduce that the right hand side of \eqref{Eq.BoundMaxMoments} is 
 $ \ll J/L^2.$
 We now choose $L=(\log\log x)^2$. This implies that the proportion of fundamental discriminants $d\in \mc{E}_1(x)$ is 
 \begin{equation}\label{Eq.Exceptionala2} \ll J (\log\log x)^{-4}\ll (\log\log x)^{-3}.
 \end{equation}

 We now handle the second condition. Let $\ep=1/(\log\log x)^2$ and $\mc{E}_2(x)$ be the set of fundamental discriminants $d\in \Ds$ such that $\min_{j\leq J} |\Ld(z_j)/V_j| \leq \ep$. Then by Theorem \ref{Theorem: Discrepancy} we obtain 
\begin{equation}\label{Eq.Exceptionala3}
\begin{aligned}
\frac{|\mc{E}_2(x)|}{|\Ds|}&= \frac{1}{|\Ds|}\bigg|\bigcup_{j=1}^J\Big\{d\in \Ds : \Ld(z_j)/V_j \in [-\ep, \ep] \Big\}\bigg|\\
& \leq \sum_{j=1}^J
\frac{1}{|\Ds|}\bigg|\Big\{d\in \Ds : \Ld(z_j)/V_j \in [-\ep, \ep] \Big\}\bigg|
\\
& \ll \sum_{j=1}^J \left(\pr\big(\Lrand(z_j)/V_j \in [-\ep, \ep]\big) + \frac{\sqrt{V_j\log\big(\log x/V_j \big)}}{\sqrt{\log x}}\right)\\
&\ll \frac{1}{\log\log x}+ \sum_{j=1}^J\frac{\sqrt{3^j\log\big(\log x/3^j \big)}}{\sqrt{\log x}},
\end{aligned}
\end{equation}
by \eqref{Eq.BoundSmallRandom}. To bound the sum over $j$ we split it in two parts $1\leq j\leq J_0$ and $J_0<j\leq J$, where $J_0= \lfloor\frac{1}{\log 3} (\log\log x-4\log \nu(x))\rfloor$. In the first part we use that $\log(\log x/3^j)\leq (\log x/3^j)^{1/2}$, while for the second we use that $\log(\log x/3^j)\ll \log \nu(x). $ This implies 
\begin{align*}
\sum_{j=1}^J\frac{\sqrt{3^j\log\big(\log x/3^j \big)}}{\sqrt{\log x}}\ll \sum_{1\leq j\leq J_0}\left(\frac{3^j}{\log x}\right)^{1/4}+ \sqrt{\frac{\log \nu(x)}{\log x}}\sum_{J_0<j\leq J} 3^{j/2} \ll \sqrt{\frac{\log \nu(x)}{\nu(x)}}.
\end{align*}
Inserting this bound in \eqref{Eq.Exceptionala3} shows that $|\mc{E}_2(x)|\ll x \sqrt{\log \nu(x)}/\sqrt{\nu(x)}$.  To finish the proof, we let $\D_2(x)= \Ds\setminus(\mc{E}_1(x)\cup \mc{E}_2(x)).$ Then combining our estimate on $\mc{E}_2(x)$ with \eqref{Eq.Exceptionala1} and \eqref{Eq.Exceptionala2} we deduce that 
$$ |\Df\setminus \D_2(x)|\ll x\sqrt{\frac{\log \nu(x)}{\nu(x)}}, $$
and for all $d\in \D_2(x)$ we have  $\max_{j\leq J}  M_{j, d}/V_j \leq (\log\log x)^2$ and   
$\min_{j\leq J} |\Ld(z_j)|/V_j\geq (\log\log x)^{-2}.$ Thus, if $d\in \D_2(x)$  then \eqref{Eq.Jensen}  implies that the number of real zeros of $\Ld$ on $\mathcal{I}$ is 
$$ \ll J (\log \log \log x)\ll (\log\log x)(\log\log \log x),$$
as desired.
\end{proof}

\section{Real zeros of $L'(s, \chi_d)$ near the central point}\label{Sec. Thm main near the central point}
 In this section we prove the following result, which together with Theorem \ref{Thm:RealZerosAway1/2} imply Theorem \ref{Thm:Main}.
\begin{thm}\label{Thm:ZerosNearCentral}
Let $
s_0 = 1/2 + (\log_2 x)/(\log x\log_3 x),
$ and put
$r = 2s_0 - 1.
$
Then for all $d \in \D(x)$, except for a set of cardinality
$
\ll x \log_3 x/\sqrt{\log\log x},
$
the number of zeros of $L'(s,\chi_d)$ inside the circle centered at $s_0$
with radius $r$ is
$
\ll \log\log x.
$
\end{thm}
To prove this result we will need two technical results. The first is an upper bound for the second moment of $L^\prime(s, \chi_d)$ at points $s$ near $1/2$.

\begin{pro}\label{Pro:SecondMomentL'}
Let $\nu$ be a positive function such that $\nu(x)\leq \log\log x$ for large $x$ and $\nu(x)\to\infty$ as $x\to \infty$. 
Then, uniformly for all complex numbers $\alpha$ such that $|\alpha|\le \nu(x)/\log x$ we have 
$$
\sum_{d\in\mathcal D(x)} \left|L'\Big(\frac12+\alpha,\chi_d\Big)\right|^2
\;\ll\;
x(\log x)^{5}e^{4\nu(x)},
$$
where the implicit constant is absolute.
\end{pro}
\begin{rem}
 It is worth emphasizing that the exponent of $\log x$ in the above upper bound
is best possible.
Indeed, Jutila \cite{Ju81} proved that the second moment of $L(1/2,\chi_d)$
is asymptotic to $c_0(\log x)^3$ for some positive constant $c_0$.
Moreover, the ``recipe'' for computing moments of $L$-functions developed
in \cite{CFKRS} predicts that one should gain an additional factor of
$(\log x)^2$ in passing from the second moment of $L(1/2,\chi_d)$ to that of $L'(1/2,\chi_d)$
(see also \cite{CRS06}, where the authors conjecture  asymptotic formulas for the  moments of $|\zeta'(1/2+it)|$).
This prediction is further consistent with random matrix theory, since our
family is of symplectic type.
\end{rem}
To prove Proposition \ref{Pro:SecondMomentL'}, one may proceed in a classical way,
using the approximate functional equation for $L(s,\chi_d)$ (which can then be
differentiated) together with the Poisson summation formula, following earlier
works on low moments of $L(1/2,\chi_d)$ (see, for example, \cite{Ju81} and \cite{Soun00}).
However, since we only aim for an upper bound, we found a considerably more streamlined
proof by relying instead on Armon’s bound (see Lemma \ref{Lem:Armon} above) for the second moment of character sums.

\begin{proof}[Proof of Proposition \ref{Pro:SecondMomentL'}]
Let $d\in \D(x)$. For $\re(s)>0$, we have by partial summation
\begin{equation}\label{Eq:mellinL}
L(s,\chi_d)
=
s\int_1^\infty \frac{S_d(u)}{u^{1+s}} \,du
\end{equation}
where 
$$
S_d(u):=\sum_{n\le u}\chi_d(n).
$$
Moreover, the integral defines an analytic function on the half plane $\re(s)>0$, since 
\begin{equation}\label{Eq:Polya_Vinogradov}
 S_d(u) \ll \sqrt{d}\log d\ll \sqrt{x}\log x,
\end{equation}
for all real numbers $u\geq 1$ by the Pólya--Vinogradov inequality. 
Differentiating both sides of \eqref{Eq:mellinL} with respect to $s$ yields
\begin{equation}\label{Eq:mellin-derivative}
L'(s,\chi_d)
=
\int_1^\infty \frac{S_d(u)}{u^{1+s}}\,du
-
s\int_1^\infty \frac{S_d(u) \log u}{u^{1+s}}\,du.
\end{equation}
Let 
$U=x^2,$ and put $s=1/2+ \alpha$ and $\sigma=\re (s).$ 
Using \eqref{Eq:Polya_Vinogradov} we have
\begin{align*}
\int_U^\infty \frac{S_d(u)}{u^{1+s}}\,du
-
s\int_U^\infty \frac{S_d(u) \log u}{u^{1+s}}\,du
\ll
\sqrt x\log x \int_U^\infty \frac{\log u}{u^{1+\sigma}}\,du
\ll x^{-1/3},
\end{align*}
by our assumption on $\alpha$. Therefore, we deduce that 
$$ |L'(s, \chi_d)|\ll \int_1^U \frac{|S_d(u)|\log(u+1)}{u^{1+\sigma}} du+x^{-1/3}.$$
Summing over $d\in\D(x)$, expanding the square, and exchanging
summation and integration, we obtain
\begin{equation}\label{Eq:BoundSecondML'}
\sum_{d\in\mathcal D(x)} |L'(s,\chi_d)|^2
\ll 
\int_1^U\!\!\int_1^U
\frac{\log(u+1)\log(v+1)}{\,
u^{1+\sigma}v^{1+\sigma}}
\sum_{d\in\mathcal D(x)} |S_d(u)S_d(v)|
\,du\,dv + x^{1/3}.
\end{equation}
We now use Lemma \ref{Lem:Armon} and the Cauchy-Schwarz inequality to get
$$
\sum_{d\in\mathcal D(x)} |S_d(u)S_d(v)|
\le
\Big(\sum_{d\in \D(x)} |S_d(u)|^2\Bigr)^{1/2}
\Bigl(\sum_{d\in \D(x)}|S_d(v)|^2\Big)^{1/2}
\ll
x(uv)^{1/2}\log x, 
$$
for all $u, v\geq 1.$ Inserting this estimate  into \eqref{Eq:BoundSecondML'} we derive
\begin{align*}
\sum_{d\in\mathcal D(x)} |L'(s,\chi_d)|^2
&\ll
x\log x
\left(\int_1^U
\frac{\log(u+1)}{
u^{\frac{1}{2}+\sigma}}
\,du\right)^2+x^{1/3}\\
& \ll x(\log x) \,  U^{2\nu(x)/\log x}
\left(\int_1^U
\frac{\log(u+1)}{
u}
\,du\right)^2+x^{1/3}\\
& \ll x(\log x)^{5}e^{4\nu(x)},
\end{align*}
since $\sigma\geq 1/2-\nu(x)/\log x$. This completes the proof.
\end{proof}

Next, we establish the following large deviation result for $\log |L(s, \chi_d)|$ for $s$ close to the half-line.
\begin{pro}\label{Pro:DistribLogL}
Let $s_0=1/2+(\log_2 x)/(\log x\log_3 x)$. Define
\begin{align*}
\mathcal{D}_1(x):=\{d\in \mathcal{D}(x)\colon \log |L(s_0, \chi_d)|>(\log \log x)/4\}.
\end{align*}
Then, we have
\begin{align*}
|\mathcal{D}(x)\setminus \mathcal{D}_1(x)|\ll \dfrac{x}{\log \log x}.
\end{align*}
\end{pro}

\begin{proof}
  Let $y=\exp(4\log x\log_3 x/\log_2x)$ so that $s_0= 1/2+4/\log y$. We also put 
 \begin{align*}
\mathcal{D}_2(x)=\{d\in \mathcal{D}(x)\colon \sigma_{y, d}=1/2+4/\log y\},
\end{align*}
where   $\sigma_{y, d}$ is given by \eqref{Def. sigma_y, d}. By Lemma \ref{Lemma: exceptional d}, there exists a constant $C_0>0$ such that
\begin{align}\label{Eq: exceptional set i}
    |\mathcal{D}(x)\setminus\mathcal{D}_2(x)|\ll x\exp\bigg(-C_0\dfrac{\log_2 x}{\log_3 x}\bigg).
\end{align}
Let $d\in \mathcal{D}_2(x)$. Then, by \eqref{Eq.ApproxLogLDirichlet}, we have
\begin{align*}
\log |L(s_0, \chi_d)|=\sum_{n\leq y^3}\dfrac{\Lambda_{y, d}(n)}{n^{s_0}\log n} + O\bigg(\dfrac{1}{\log y}\bigg|\sum_{n\leq y^3}\dfrac{\Lambda_{y, d}(n)}{n^{s_0}}\bigg| + \dfrac{\log x}{\log y}\bigg).
\end{align*}
By the definition of $\Lambda_{y, d}$ from \eqref{Eq.DefLambdaYD}, we note that
\begin{align*}
\sum_{n\leq y^3}\dfrac{\Lambda_{y, d}(n)}{n^{s_0}\log n}=\sum_{n\leq y}\dfrac{\Lambda(n)\chi_d(n)}{n^{1/2}\log n} + \sum_{n\leq y}\dfrac{\Lambda(n)\chi_d(n)}{n^{1/2}\log n}(n^{-4/\log y}-1) + \sum_{y<n\leq y^3}\dfrac{\Lambda_{y, d}(n)}{n^{s_0}\log n}.
\end{align*}
By Mertens theorem, the first term in the right-hand side of the above expression can be simplified as
\begin{align*}
\sum_{n\leq y}\dfrac{\Lambda(n)\chi_d(n)}{n^{1/2}\log n}&=\sum_{p\leq y}\dfrac{\chi_d(p)}{p^{1/2}} + \dfrac{1}{2}\sum_{\substack{p\leq \sqrt{y}\\ p\nmid 2d}}\dfrac{1}{p} + O(1)\\
& =\sum_{p\leq y}\dfrac{\chi_d(p)}{p^{1/2}}+ \frac12\log \log x + O(\log_3x),
\end{align*}
since the contribution of  prime powers $p^k$ with $k\geq 3$ is bounded, and $\sum_{p| 2d} 1/p\ll \log_3 |d|.$
Therefore, for all $d\in \D_2(x)$ we have
\begin{align}\label{Eq:ApproxLogLD2}
\log |L(s_0, \chi_d)| &-\dfrac{1}{2}\log \log x\nonumber\\
= &\: \sum_{p\leq y}\dfrac{\chi_d(p)}{p^{1/2}} + \sum_{n\leq y}\dfrac{\Lambda(n)\chi_d(n)}{n^{1/2}\log n}(n^{-4/\log y}-1) + \sum_{y<n\leq y^3}\dfrac{\Lambda_{y, d}(n)}{n^{s_0}\log n}\\
&+ O\bigg(\dfrac{1}{\log y}\bigg|\sum_{n\leq y^3}\dfrac{\Lambda_{y, d}(n)}{n^{s_0}}\bigg| + \dfrac{\log\log  x}{\log_3 x}\bigg).\nonumber
\end{align}
Let $\D_3(x)$ and $\D_4(x)$ be the subsets of discriminants $d\in\D(x)$ such that 
$$
\bigg|\sum_{p\leq y}\dfrac{\chi_d(p)}{p^{1/2}} + \sum_{n\leq y}\dfrac{\Lambda(n)\chi_d(n)}{n^{1/2}\log n}(n^{-4/\log y}-1) + \sum_{y<n\leq y^3}\dfrac{\Lambda_{y, d}(n)}{n^{s_0}\log n}\bigg|\leq \frac{1}{5}\log\log x,
$$
and 
$$\frac{1}{\log y}\bigg|\sum_{n\leq y^3}\dfrac{\Lambda_{y, d}(n)}{n^{s_0}}\bigg|\leq \sqrt{\log\log x},$$
respectively. 
Then, by \eqref{Eq:ApproxLogLD2} we observe that 
$\D_2(x)\cap  \D_3(x)\cap \D_4(x) \subset \D_1(x)\cap \D_2(x).$ By Markov's inequality we have
\begin{align}\label{Eq: Markov app}
|\mathcal{D}(x)\setminus \D_3 (x)|\ll \dfrac{1}{(\log \log x)^2}\Sigma_1(x), 
\end{align}
and 
\begin{align}\label{Eq: Markov app2}
|\mathcal{D}(x)\setminus \D_4 (x)|\ll \dfrac{1}{\log \log x}\Sigma_2(x), 
\end{align}
where
\begin{align*}
\Sigma_1(x):=\sum_{d\in \mathcal{D}(x)}\bigg|\sum_{p\leq y}\dfrac{\chi_d(p)}{p^{1/2}} + \sum_{n\leq y}\dfrac{\Lambda(n)\chi_d(n)}{n^{1/2}\log n}(n^{-4/\log y}-1) + \sum_{y<n\leq y^3}\dfrac{\Lambda_{y, d}(n)}{n^{s_0}\log n}\bigg|^2
\end{align*}
and
\begin{align*}
\Sigma_2(x):=\dfrac{1}{(\log y)^2}\sum_{d\in \mathcal{D}(x)}\bigg|\sum_{n\leq y^3}\dfrac{\Lambda_{y, d}(n)}{n^{s_0}}\bigg|^2.
\end{align*}
First we handle the sum $\Sigma_2(x)$. By applying Lemma \ref{Lem.LargeSieve} and arguing as in the proof of Lemma \ref{Lemma: Moment Bound}, we infer that
\begin{align*}
\Sigma_2(x)\ll \dfrac{1}{(\log y)^2}x(\log y)^2\ll x.
\end{align*}
Similarly, applying Lemmas \ref{Lem: LargeSieveKLM} and \ref{Lem.LargeSieve}, we derive
\begin{align*}
\Sigma_1(x) &\ll \sum_{d\in \mathcal{D}(x)}\bigg|\sum_{p\leq y}\dfrac{\chi_d(p)}{p^{1/2}}\bigg|^2 + \sum_{d\in \mathcal{D}(x)}\bigg|\sum_{n\leq y}\dfrac{\Lambda(n)\chi_d(n)}{n^{1/2}\log n}(n^{-4/\log y}-1)\bigg|^2  \\
&  \quad \quad \quad \quad \quad +\sum_{d\in \mathcal{D}(x)}\bigg|\sum_{y<n\leq y^3}\dfrac{\Lambda_{y, d}(n)}{n^{s_0}\log n}\bigg|^2\\
&\ll x\log \log x.
\end{align*}
Hence, combining the above estimates for $\Sigma_1(x)$ and $\Sigma_2(x)$ together with the relations \eqref{Eq: Markov app} and \eqref{Eq: Markov app2} we obtain
\begin{align*}
|\mathcal{D}(x)\setminus (\D_3 (x)\cap \D_4(x))|\ll \dfrac{x}{\log \log  x}.
\end{align*}
Finally, we apply the above estimate together with the relation \eqref{Eq: exceptional set i} to conclude that
\[|\mathcal{D}(x)\setminus \mathcal{D}_1 (x)|\ll \dfrac{x}{\log \log  x},\]
as desired.
\end{proof}

Given the above two propositions, we are now ready to prove Theorem \ref{Thm:ZerosNearCentral}.

\begin{proof}[Proof of Theorem \ref{Thm:ZerosNearCentral}]
Let $\mc{C}_1, \mc{C}_2$ and $\mc{C}_3$ be the circles centered at $s_0$ with radii
$r_1 = r$, $r_2 = 2r$ and $r_3 = 3r$ respectively.
By Jensen's formula, the number of zeros of $L'(s,\chi_d)$ inside $\mathcal{C}_1$ is
\begin{equation}\label{Eq:JenssenCriticalLine}
\quad
\le \frac{1}{\log 2}
\left[
\log\left(
\max_{s\in \mc{C}_2} |L'(s,\chi_d)|
\right)
-
\log\left(
|L'(s_0,\chi_d)|
\right)
\right].
\end{equation}
Let $\D_1(x)$ be the set of discriminants $d\in \D(x)$ such that
$$
\log |L(s_0,\chi_d)| > \frac14 \log\log x.
$$
Then it follows from Proposition \ref{Pro:DistribLogL} that
$
|\D(x)\setminus \D_1(x)| \ll x/\log\log x.
$
We now denote by $\D_5(x)$ the set of $d\in \D(x)$ such that
$$
\left|\frac{L'}{L}(s_0,\chi_d)\right|
>
\frac{\log x}{(\log\log x)^2}.
$$
Combining Theorem \ref{Theorem: Discrepancy} with Lemma \ref{Lemma: exceptional d} and \eqref{Eq.BoundSmallRandom} we obtain
\begin{align*}
|\D(x)\setminus \D_5(x)|\ll 
x\cdot \pr\Big(|\Lrand(s_0)/V_{s_0}|\leq  1/\log\log x\Big) + x\frac{\log_3 x}{\sqrt{\log\log x}}
\ll \frac{x\log_3 x}{\sqrt{\log\log x}}.
\end{align*}
Putting these estimates together, we deduce that for all
$d \in \D_1(x)\cap \D_5(x)$ we have
$$
|L'(s_0,\chi_d)|
=
\left|\frac{L'}{L}(s_0,\chi_d)\right|
\exp\left(\log |L(s_0,\chi_d)|\right)
>
(\log x)^{\frac65}.
$$
Next, we define $\D_6(x)$ to be the set of $d\in \D(x)$ such that
$$
\max_{s\in \mc{C}_2} |L'(s,\chi_d)| \le (\log x)^3.
$$
Let $s\in \mc{C}_2$. Since $L'(z,\chi_d)^2$ is entire, it follows from
Cauchy's formula that
$$
L'(s,\chi_d)^2
=
\frac{1}{2\pi i}
\int_{z\in \mc{C}_3}
\frac{L'(z,\chi_d)^2}{z-s}\,dz.
$$
Thus,
$$
\max_{s\in \mc{C}_2} |L'(s,\chi_d)|^2
\ll
\frac{1}{r}
\int_{z\in \mc{C}_3} |L'(z,\chi_d)|^2\,|dz|,
$$
since $|z-s| \geq |z-s_0|-|s-s_0|= r$ for all $s\in \mc{C}_2$ and $z\in \mc{C}_3$. Therefore, it follows from Markov's inequality and Proposition \ref{Pro:SecondMomentL'} that
\begin{align*}
|\D(x)\setminus \D_6(x)|
&\le
\frac{1}{(\log x)^6}
\sum_{d\in \D(x)}
\max_{s\in \mc{C}_2} |L'(s,\chi_d)|^2\\
&\ll
\frac{1}{(\log x)^6\,r}
\int_{z\in \mc{C}_3}
\sum_{d\in \D(x)} |L'(z,\chi_d)|^2\,|dz|\\
&\ll
x(\log x)^{-1+o(1)},
\end{align*}
since
$
\int_{z\in \mc{C}_3} |dz| \asymp r.
$ Finally, we let
$
\D_7(x) = \D_1(x)\cap \D_5(x)\cap \D_6(x).
$
Combining the above estimates we obtain
$$
|\D(x)\setminus \D_7(x)|
\ll
x\,\frac{\log_3 x}{\sqrt{\log\log x}},
$$
and for all $d\in \D_7(x)$ we have
$$
\log\left(
\max_{s\in \mc{C}_2} |L'(s,\chi_d)|
\right)
-
\log\left(
|L'(s_0,\chi_d)|
\right)
\leq 
2\log\log x.
$$
Inserting this bound in \eqref{Eq:JenssenCriticalLine} completes the proof.
\end{proof}

\section{The location of real zeros of $L'(s, \chi_d)$: Proof of Theorem \ref{Thm:ConditionalZerosNear1/2}}\label{Section:MomentsNearHalf}
Let $d\in \Df$ and recall that $\Ld(s)=-\frac{L'}{L}(s, \chi_d)$. The completed $L$-function associated to $L(s, \chi_d)$ is 
$$ \Lambda(s, \chi_d)= \left(\frac{d}{\pi}\right)^{s/2}\Gamma \left(\frac{s}{2}\right)L(s, \chi_d),$$
since $\chi_d(-1)=1$. The completed $L$-function satisfies the self-dual functional equation 
$$ \Lambda(s, \chi_d)= \Lambda(1-s, \chi_d),$$ and its zeros are precisely the non-trivial zeros of $L(s, \chi_d)$. We start by recording the following standard lemma. 

\begin{lem}\label{lem.Hadamard}
Let $s\in \mathbb{C}$ be such that $1/4<\re(s)\leq 5/4$, and $s$ does not coincide with a non-trivial zero of $L(s, \chi_d)$. Then we have 
\begin{equation}\label{Eq.Hadamard}
\Ld(s)= \frac{1}{2}\log\left(\frac{d}{\pi}\right)+\frac12 \frac{\Gamma'}{\Gamma} \left(\frac{s}{2}\right)- \sum_{\rho}\frac{1}{s-\rho},
\end{equation}
where the sum is over all non-trivial zeros of $L(s, \chi_d)$. We also have 
\begin{equation} \label{Eq.DerivativeL'}  
 (\Ld)^{\prime}(s)= \sum_{\rho}\frac{1}{(s-\rho)^2}+O(1). 
\end{equation}
\end{lem}
\begin{proof}
   The identity \eqref{Eq.Hadamard} follows from the Hadamard product formula for $\Lambda(s, \chi_d)$ (see for example Eq. (17) and (18) of \cite[Chapter 12]{Da80}). While the second estimate follows by taking the derivative of \eqref{Eq.Hadamard} with respect to $s$.
\end{proof}

Throughout this section we let 
\begin{align}\label{Eq: choice of nu and nu prime}
\nu(x):=\frac{(\log\log x)^{1/2-\delta}}{\log_3 x} \quad \text{ and } \quad \widetilde{\nu}(x):=\frac{(\log\log x)^{1/2+\delta}}{\sqrt{\log_3 x}}, 
\end{align}
where $0<\delta<1/2$ is the constant in Assumption 2. We also put $y=x^{4/\nu(x)}$ and let $\D_y(x)$ be the set in the statement of Lemma \ref{Lemma: exceptional d}, namely 
$$\mc{D}_y(x)=\{d\in \mc{D}(x)\colon
\sigma_{y, d} = 1/2+4/\log y\}.$$
Then it follows from Lemma \ref{Lemma: exceptional d} that $|\D(x)\setminus \D_y(x)|\ll xe^{-C_0\nu(x)}$, for some positive constant $C_0.$

\begin{pro}\label{pro.MomentsNear1/2}
Let $s_0=1/2+\nu(x)/\log x$ and $D_1$, $D_2$ be the discs of center $s_0$ and radii $R_1=s_0- 1/2+1/(2\widetilde{\nu}(x)\log x)$ and $R_2= s_0- 1/2+1/(\widetilde{\nu}(x)\log x)$ respectively. Let $\widetilde{\D}_0(x)$ be the set of discriminants $d\in \D_y(x)$ such that $L(s, \chi_d)$ is free of zeros inside the disc $D_2$.  Then, uniformly for all $s\in D_1$ we have 
   $$ \frac{1}{|\Df|}\sum_{d\in \widetilde{\D}_0(x)} |\Ld(s)|^2 \ll   (\log x\log\log x)^2.$$
\end{pro}

\begin{proof}
 Let $d\in \widetilde{\D}_0(x)$. Then $\sigma_{y, d}=s_0$. 
Moreover, by \eqref{Eq.ApproxL_dDirichlet}, we have 
\begin{equation}\label{Eq.BoundL'Primes}
\Ld(\sigma_{y, d}) \ll \log d +|A_d(y)|,
\end{equation}
where $$ A_d(y):= \sum_{n\leq y^3} \frac{\Lambda_{y, d}(n)}{n^{\sigma_{y, d}}},$$
 and $\Lambda_{y,  d}$ is given by \eqref{Eq.DefLambdaYD}.
Let $s\in D_1$. Since $L(s, \chi_d)$ is free of zeros inside the disc $D_2$ we get 
\begin{equation}\label{Eq.LLZeroHypothesis}
    \min_{\rho}|s-\rho|\gg \frac{1}{\widetilde{\nu}(x)\log x},
\end{equation}
where the minimum runs over the non-trivial zeros of $L(s, \chi_d)$. Furthermore, using  
the identity
$$ -\frac{1}{s-\rho}= -\frac{1}{\sigma_{y, d}-\rho}+\frac{s-\sigma_{y, d}}{(\sigma_{y, d}-\rho)^2}+\frac{(s-\sigma_{y, d})^2}{(\sigma_{y, d}-\rho)^2(s-\rho)}.
$$ together with \eqref{Eq.Hadamard} and   \eqref{Eq.DerivativeL'} we obtain\footnote{A similar estimate was derived by Selberg  for the Riemann zeta function, see Eq. (12) of \cite{SelUnp}.}
$$ 
\Ld(s)= \Ld(\sigma_{y, d})
+ (s-\sigma_{y, d}) (\Ld)^{\prime}(\sigma_{y, d})+ \sum_{\rho}\frac{(s-\sigma_{y, d})^2}{(\sigma_{y, d}-\rho)^2(s-\rho)}+O(1),
$$
where $\rho$ runs over the non-trivial zeros of $L(s, \chi_d)$. Therefore, combining \eqref{Eq.SelbergZerosPrimes}, \eqref{Eq.DerivativeL'},  \eqref{Eq.BoundL'Primes} and \eqref{Eq.LLZeroHypothesis} we get 
$$ |\Ld(s)|\ll (\log d+ |A_d(s)|)\left(1+\frac{|s-\sigma_{y, d}|}{\sigma_{y, d}-1/2}+ \frac{|s-\sigma_{y, d}|^2 \widetilde{\nu}(x)\log x}{\sigma_{y, d}-1/2} \right).$$
Since $|s-\sigma_{y, d}|\ll \nu(x)/\log x= \sigma_{y, d}-1/2$ we deduce that 
$$
|\Ld(s)|\ll (\log x+ |A_d(s)|) \log\log x.
$$
Finally, by the same calculation leading to \eqref{Eq.BoundMomentsLdSpecial} we infer from Lemma \ref{Lem.LargeSieve} that  
\begin{align*}
\frac{1}{|\mathcal{D}(x)|}\sum_{d\in \widetilde{\D}_0(x)}|\Ld(s)|^{2} & \ll  (\log x \log\log x)^{2}+ (\log\log x)^{2}\frac{1}{|\mathcal{D}(x)|}\sum_{d\in \mathcal{D}(x)}|A_d(s)|^{2}\\
&
\ll  (\log x\log\log x)^2.
\end{align*}
This completes the proof.
\end{proof}

\begin{proof}[Proof of Theorem \ref{Thm:ConditionalZerosNear1/2}]
Let $s_0= 1/2+\nu(x)/\log x$. We consider the concentric circles $\mc{C}_0$, $\mc{C}_1$, $\mc{C}_2$, and $\mc{C}_3$ of center $s_0$ and radii $r_0$, $r_1$,  $r_2$, and $r_3$ respectively, where $r_0=s_0-1/2$, $r_1= r_0+ 1/(4\widetilde{\nu}(x)\log x)$,  $r_2= r_0+ 1/(2\widetilde{\nu}(x)\log x)$, and $r_3= r_0+ 3/(4\widetilde{\nu}(x)\log x)$.      
 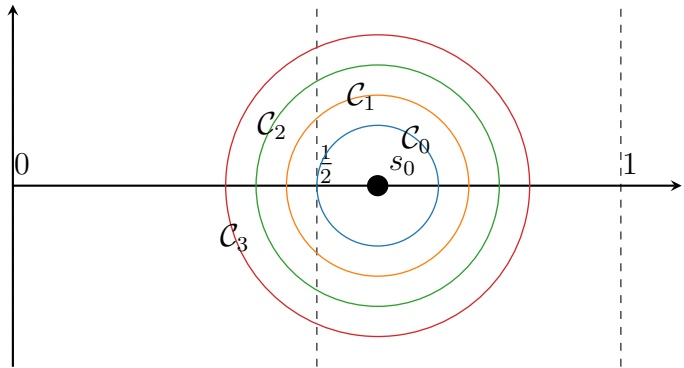
\begin{figure}[h]
  \centering

\begin{tikzpicture}[>=stealth,scale=8]

  \draw[->,thick] (0,0) -- (1.1,0) node[right] {};
  \draw[->,thick] (0,-0.3) -- (0,0.3) node[above] {};


  \foreach \x/\label in {0/0,0.5/{\frac{1}{2}},1/1}{
    \draw[black,dashed] (\x,-0.3) -- (\x,0.3); 
    \coordinate (top) at (\x,0.01);
    \node at ($(top) + (60:0.03)$) {$\label$};    
}

  \def\r{0.1}           
  \def\c{\r + 0.5}      
  \coordinate (C) at (\c,0);

  \def\rzero{\r}         
  \def\rone{\r + 0.05}    
  \def\rtwo{\r + 0.1}    
  \def\rthree{\r + 0.15}  

  \definecolor{col0}{RGB}{31,119,180} 
  \definecolor{col1}{RGB}{255,127,14} 
  \definecolor{col2}{RGB}{44,160,44}  
  \definecolor{col3}{RGB}{214,39,40}  

  \draw[line width=0.5pt,col0] (C) circle (\rzero);
  \draw[line width=0.5pt,col1] (C) circle (\rone);
  \draw[line width=0.5pt,col2] (C) circle (\rtwo);
  \draw[line width=0.5pt,col3] (C) circle (\rthree);

   Mark center
  \fill (C) circle (0.5pt) node[above right] {$s_0$};

   \node at ($(C) + (50:\rzero)$) {$\mathcal{C}_0$};
  \node at ($(C) + (100:\rone)$) {$\mathcal{C}_1$};
  \node at ($(C) + (150:\rtwo)$) {$\mathcal{C}_2$};
  \node at ($(C) + (200:\rthree)$) {$\mathcal{C}_3$};
  \end{tikzpicture}

  \caption{Four concentric circles $\mathcal{C}_0$, $\mathcal{C}_1$, $\mathcal{C}_2$, and $\mathcal{C}_3$.}
  \label{fig: circle 2}
\end{figure}

Let $\widetilde{\mc{D}}_0(x)$ be the set of discriminants in the statement of Proposition \ref{pro.MomentsNear1/2}.  Note that the disc of center $s_0$ and radius $r_0+ 1/(\widetilde{\nu}(x)\log x)$ is included in the rectangle $\mc{R}=\{s \in \mathbb{C}: 1/2-1/\log x\leq \re(s)\leq 1$ and $|\im(s)|\leq \sqrt{\log\log x}/\log x\}$, and that the intersection of this disc with the critical line is the vertical segment $\{1/2+it, \, |t| \leq \eta\}$, where $$\eta \asymp \frac{\sqrt{\nu(x)}}{\sqrt{\widetilde{\nu}(x)}\log x}= o\left(\frac{1}{(\log\log x)^{\delta}\log x}\right). $$ Therefore, by Assumptions 1 and 2 we have $|\D(x)\setminus \widetilde{\mc{D}}_0(x)|= o(x)$. Let $d\in \widetilde{\mc{D}}_0(x)$. Then $\Ld$ is analytic inside the circle $\mc{C}_3$ and hence by Jensen's formula the number of real zeros of $\Ld$ in the interval $[1/2, 1/2+\nu(x)/\log x]$ is bounded by
\begin{equation}
\label{Eq.JensenNear1/2}
\begin{aligned}
&\frac{\log\big(\max_{s\in \mathcal{C}_1}|\Ld(s)|/|\Ld(s_0)|\big)}{\log (r_1/r_0)}\\
&=\frac{1}{\log (r_1/r_0)}\left(\log\big(\max_{s\in \mathcal{C}_1}|\Ld(s)|/\log x\big)- \log\big(|\Ld(s_0)|/\log x\big)\right),
\end{aligned}
\end{equation} 
since $[1/2, 1/2+\nu(x)/\log x]\subset \{z\in \mathbb{C}: |z-s_0|\leq r_0\}.$ 
Moreover, by Cauchy's formula, for all $s\in \mathcal{C}_1$, we have
$$\Ld(s)^{2}= \frac{1}{2\pi i} \int_{z\in \mc{C}_2} \frac{\Ld(z)^{2}}{z-s} dz. $$
This implies 
\begin{equation}\label{Eq.Cauchy2}
\max_{s\in \mc{C}_1} |\Ld(s)|^{2} \ll \widetilde{\nu}(x)\log x \int_{z\in \mc{C}_2} |\Ld(z)|^{2} |dz|, 
 \end{equation}
 since $|z-s|\geq r_2-r_1= 1/(4\widetilde{\nu}(x)\log x)$ for all $z\in \mc{C}_2$ and $s\in \mc{C}_1$. By  Proposition \ref{pro.MomentsNear1/2} we have 
\begin{equation}\label{Eq.SecondMomentNear1/2}
 \frac{1}{|\Df|}\sum_{d\in \widetilde{\D}_0(x)} |\Ld(z)|^{2} \ll  (\log x\log\log x)^2,
\end{equation}
 uniformly for all $z\in \mc{C}_2$.  Moreover, combining \eqref{Eq.Cauchy2} and \eqref{Eq.SecondMomentNear1/2} we get 
 \begin{equation}\label{Eq.BoundSecondMomentMax}\begin{aligned}
 \frac{1}{|\Df|}\sum_{d\in \widetilde{\D}_0(x)}\max_{s\in \mc{C}_1} |\Ld(s)|^{2} &\ll \widetilde{\nu}(x)\log x\int_{z\in \mc{C}_2} \frac{1}{|\Df|}\sum_{d\in \widetilde{\D}_0(x)}|\Ld(z)|^{2} |dz|\\
 & \ll (\log x)^2(\log\log x)^3,
\end{aligned}
\end{equation}
since $\int_{z\in \mc{C}_2}|dz|\asymp \nu(x)/\log x$. We now define $\mc{E}_3(x)$ to be the set of  discriminants $d\in \widetilde{\D}_0(x)$ such that $\max_{s\in \mathcal{C}_1}|\Ld(s)|/\log x\geq (\log\log x)^2$. By Markov's inequality and \eqref{Eq.BoundSecondMomentMax} we obtain
\begin{equation}\label{Eq.BoundExceptional1Near}
\frac{|\mc{E}_3(x)|}{|\Df|} \leq \frac{1}{ (\log x)^2(\log\log x)^4}\frac{1}{|\Df|}\sum_{d\in \widetilde{\D}_0(x)}\max_{s\in \mc{C}_1} |\Ld(s)|^{2}\ll \frac{1}{\log\log x}.
\end{equation}
Next, we let $\mc{E}_4(x)$ be the set of discriminants $d\in \widetilde{\D}_0(x)$ such $|\Ld(s_0)|/\log x\leq \ep/\nu(x)$ where $ \ep=1/\log\log x.$ Then it follows from Theorem \ref{Theorem: Discrepancy} together with \eqref{Eq.BoundSmallRandom} that 
\begin{align*}
\frac{|\mc{E}_4(x)|}{|\Df|}
& = 
\frac{1}{|\Df|}\bigg|\Big\{d\in \widetilde{\D}_0(x) : \Ld(s_0)/V_{s_0} \in [-\ep, \ep] \Big\}\bigg|
\\
& \ll \Big(\pr\big(\Lrand(s_0)/V_{s_0} \in [-\ep, \ep]\big)\Big) + \sqrt{\frac{\log \nu(x)}{\nu(x)} }
\ll \sqrt{\frac{\log \nu(x)}{\nu(x)} }.
\end{align*}
Finally, we let $\widetilde{\D}_1(x)=\widetilde{\D}_0(x)\setminus\big(\mc{E}_3(x) \cup  \mc{E}_4(x)\big)$. 
Then we deduce from the above that $|\D(x)\setminus \widetilde{\D}_1(x)|=o(x)$. Moreover, by \eqref{Eq.JensenNear1/2}, for all $d\in \widetilde{\D}_1(x)$, the number of zeros of $\Ld$ in the interval $[1/2, 1/2+\nu(x)/\log x]$ is 
$$ \ll \frac{\log_3x}{\log(r_1/r_0)}\ll \nu(x)\widetilde{\nu}(x)\log_3x \ll \frac{\log\log x}{\sqrt{\log_3 x}},
$$
by our choice of $\nu(x)$ and $\widetilde{\nu}(x)$ in \eqref{Eq: choice of nu and nu prime}. Combining this estimate with  \eqref{Eq.LowerBoundKLM} completes the proof. 
\end{proof}

\section*{Acknowledgments}

YL is supported by a junior chair of the Institut Universitaire de France.

\end{document}